\documentclass[12pt,letterpaper,titlepage]{amsart}
\usepackage{amsmath, amssymb, amsthm, amsfonts,amscd,amsaddr, xr}
\usepackage[all]{xy}
\usepackage[usenames]{color}
\usepackage[in]{fullpage}

\theoremstyle{plain}
\newtheorem{theorem}{Theorem}[section]

\newtheorem{con}[theorem]{Conjecture}
\newtheorem{problem}[theorem]{Problem}
\newtheorem{prop}[theorem]{Proposition}
\newtheorem{lemma}[theorem]{Lemma}
\newtheorem{definition}[theorem]{Definition}
\theoremstyle{definition}
\newtheorem{ex}[theorem]{Example}
\newtheorem{rmk}[theorem]{Remark}
\numberwithin{equation}{section}

\newcommand{\bs}{\backslash}

\newcommand{\C}{\mathbb{C}}

\newcommand{\Hc}{\mathcal{H}}

\newcommand{\Q}{\mathbb{Q}}
\newcommand{\Z}{\mathbb{Z}}
\newcommand{\Zc}{\mathcal{Z}}

\newcommand{\R}{\mathbb{R}}
\newcommand{\N}{\mathbb{N}}

\newcommand{\Aut}{\operatorname{Aut}}
\newcommand{\Exp}{\operatorname{Exp}}

\newcommand{\Sl}{\operatorname{SL}}
\newcommand{\SO}{\operatorname{SO}}

\newcommand{\PSL}{\operatorname{PSL}}
\newcommand{\Iso}{\operatorname{Iso}}

\newcommand{\Lie}{\operatorname{Lie}}
\newcommand{\Ad}{\operatorname{Ad}}
\newcommand{\ad}{\operatorname{ad}}
\newcommand{\diag}{\operatorname{diag}}
\newcommand{\vol}{\operatorname{vol}}

\newcommand{\Spec}{\operatorname{spec}}
\newcommand{\supp}{\operatorname{supp}}
\newcommand{\Span}{\operatorname{span}}
\newcommand{\err}{\operatorname{err}}
\newcommand{\rank}{\operatorname{rank}}

\newcommand{\PSl}{\operatorname{PSl}}
\newcommand{\re}{\operatorname{Re}}

\newcommand{\vb}{\operatorname{\bf v}}
\newcommand{\rb}{\operatorname{\bf r}}
\newcommand{\Cc}{\mathcal {C}}
\newcommand{\cb}{\operatorname{\bf c}}
\newcommand{\bH}{\operatorname{\bf H}}
\newcommand{\sA}{\mathsf{A}}
\newcommand{\sP}{\mathsf{P}}
\newcommand{\sw}{\mathsf{w}}

\newcommand{\Wc}{\mathcal{W}}

\def\hat{\widehat}
\def\af{\mathfrak{a}}
\def\e{\epsilon}
\def\ff {\mathfrak{f}}
\def\gf{\mathfrak{g}}
\def\cf{\mathfrak{c}}
\def\hf{\mathfrak{h}}
\def\kf{\mathfrak{k}}
\def\lf{\mathfrak{l}}
\def\mf{\mathfrak{m}}
\def\nf{\mathfrak{n}}

\def\so{\mathfrak{so}}
\def\pf{\mathfrak{p}}
\def\qf{\mathfrak{q}}

\def\sl{\mathfrak{sl}}
\def\sp{\mathfrak{sp}}
\def\spin{\mathfrak{spin}}
\def\tf{\mathfrak{t}}
\def\uf{\mathfrak{u}}
\def\su{\mathfrak{su}}

\def\wf{\mathfrak{w}}
\def\zf{\mathfrak{z}}
\def\la{\langle}
\def\ra{\rangle}
\def\1{{\bf1}}

\def\U{\mathcal{U}}
\def\E{\mathcal{E}}
\def\B{\mathcal{B}}

\def\oline{\overline}

\def\W{\mathsf{W}}
\def\sE{\mathsf{E}} 
\def\sG{\mathsf{G}} 
\def\sF{\mathsf{F}} 
\def\HH{\mathbb{H}}

\newcounter{class}
\newcommand{\yy}{\stepcounter{class}(\arabic{class})}
\newcounter{Tabelle}
\newcommand{\Tabelle}[1]{\refstepcounter{Tabelle}Table \arabic{Tabelle}\label{#1}}
\newcommand\aas{\llap{$*$}}

\title[Lattice counting]
{Geometric counting on wavefront real spherical spaces}

\begin{document}

\date{August 28, 2017}

\author[Kr\"otz]{Bernhard Kr\"{o}tz}
\email{bkroetz@gmx.de}
\address{Universit\"at Paderborn, Institut f\"ur Mathematik\\Warburger Stra\ss e 100, 
33098 Paderborn, Deutschland}
\thanks{ The second author was partially supported by ISF 1138/10 and ERC 291612}
\author[Sayag]{Eitan Sayag}
\email{eitan.sayag@gmail.com}
\address{Department of Mathematics, Ben Gurion University of the Negev\\P.O.B. 653, Be'er Sheva 84105, Israel}
\author[Schlichtkrull]{Henrik Schlichtkrull}
\email{schlicht@math.ku.dk}
\address{University of Copenhagen, Department of Mathematics\\Universitetsparken 5, 
DK-2100 Copenhagen \O, Denmark}

\begin{abstract}  We provide $L^p$-versus $L^\infty$-bounds for eigenfunctions
on a real spherical space $Z$ of wavefront type. It is shown that these bounds imply 
a non-trivial  error term estimate for  lattice counting on $Z$. 
The paper also serves as an introduction to geometric counting on spaces of the
mentioned type.

\end{abstract}

\maketitle

\section{Introduction}\label{Sect1}

Given a space $Z$ with a discrete subset $D$ and an increasing and exhausting family
$(B_R)_{R>0}$ of compact subsets, it can be of considerable interest to know
the expected number of points from $D$ inside $B_R$ for large $R$. In general terms
this is called {\it lattice counting} on $Z$. 
This paper is about lattice counting on a homogeneous space $Z=G/H$, and methods 
from harmonic analysis to approach it.   Here $G$ is a real reductive group and $H$ is a closed 
subgroup with finitely many components and real algebraic Lie algebra. 
\par At this point we do not go into the  specifics of the lattice count on $G/H$ and refer right 
away to Sections \ref{Sect2} and 
\ref{Sect4} where we give a self-contained exposition aimed at a wide audience. The terminology 
there is essentially taken from \cite{DRS}
but we emphasize more the underlying fiber bundle structure 
$H\to G\to G/H$ and relate counting on $Z$ to counting on the total space $G$ and fiber $H$.
\par The lattice counting problem including non-trivial error terms has a positive solution 
for all symmetric spaces $Z=G/H$.  A central tool there is the so-called wavefront lemma, see \cite{EM}. The wavefront lemma however holds more generally for all 
real spherical spaces of wavefront type.  In Section \ref{wfrss} we give an introduction to real spherical 
space of wavefront type and recall the proof of the wavefront lemma from \cite{KKSS1}. 
 
\par The strive for solving the lattice counting problem triggered many interesting developments, perhaps
more interesting than the problem itself. Specifically we mention here  Selberg's trace formula for the upper half plane.  Here we wish to point out another connection to harmonic analysis.  In  
\cite{KSS} we have shown that non-trivial error terms for the lattice count are tied to $L^p$-versus 
$L^\infty$-bounds of eigenfunctions on  the non-compact space $Z$.  
\par We assume now that $Z=G/H$ is unimodular, i.e. carries a $G$-invariant  positive Radon measure. 
Following \cite{B} we measure volume growth on $Z$ via a volume weight 
$$ \vb(z) = \vol_Z(Bz) \qquad (z\in Z)$$
where $B$ is some fixed neighborhood of $\1$ in $G$.  Let $1\leq p<\infty$.
Then Bernstein's invariant Sobolev lemma
(\cite{B} key lemma p.~686,  or \cite{KS2}, Lemma 4.2) implies for all $f\in C^\infty(Z)$ that 

\begin{equation} \label{ISL} |f(z)|  \leq C \vb(z)^{-{1\over p}}  \|f\|_{p;k} \qquad (z\in Z)\end{equation} 
where $\|\cdot\|_{p;k}$ is a $k$-th Sobolev norm of the $L^p$-norm $\|\cdot\|_p$ and $k>{\dim G\over p}$.

\par We recall from \cite{KSS2} that $\vb$ is uniformly bounded from below if and only if $H$ is 
reductive in $G$. Let us assume this in the sequel.  Then we obtain in particular: 

\begin{equation} \label{ISL2}
\| f\|_\infty \leq C \|f\|_{p; k} \qquad (f\in C^\infty(Z))\, .\end{equation}
What is relevant for the lattice count are estimates in the other direction. 
We call $f\in C^\infty(Z)$ an eigenfunction if it is an eigenfunction for $\Zc(\gf)$, the center of the universal 
enveloping algebra $\U(\gf)$ of $\gf=\Lie (G)$.    
Now, given $1\leq p'<p< \infty$ we ask whether there 
exist a number
$l=l(Z)>0$ and a constant $C>0$ such that 

\begin{equation}\label{expect} \| f\|_p \leq C \|f\|_{\infty; l}\, \end{equation}
holds for all $L^{p'}$-eigenfunctions $f$ on $Z$. 
It is of independent interest to classify all homogeneous spaces 
$Z=G/H$ which feature (\ref{expect}).

\par The main result of this paper is the verification of (\ref{expect}) for all wave-front real spherical spaces (see Section \ref{Sect7}). 
Combined with the harmonic analysis approach of 
\cite{KSS} this then leads to a quantitative error bound
for the lattice count on all real spherical spaces of wave-front type in Theorem 
\ref{thm error bound}.

\par Our approach to (\ref{expect}) relies on a generalization 
of Harish-Chandra's constant-term approximation of eigenfunctions on reductive groups \cite{HC} to real spherical spaces 
of wave-front type \cite{DKS}.  In Section 6 we first treat the less technical and instructive case  
of $\rank_\R Z=1$:  we cut down the techniques from 
\cite{DKS} to the absolute necessary and provide a proof of a certain constant term approximation. Let us mention that the constant term approximation obtained in this paper differs from the one in \cite{DKS}, see 
Remark \ref{debate on ct}. Finally, in Section \ref{Sect7} we adapt the proof from the rank one case  to higher rank and obtain 
(\ref{expect}) for all wave-front real spherical spaces. 

\bigskip\noindent 
 {\it Acknowledgement:}  We wish to thank the referee for his insistence on technical detail and clarity of arguments. 
It eventually made this article much more readable.

\section{Geometric counting}\label{Sect2}

A setup for {\it geometric counting} needs: 
\begin{itemize} 
\item A locally compact space $X$. 
\item A notion of volume on $X$ given by a Radon measure $\mu$.  
\item A  discrete set $D\subset X$. 
\item An increasing and exhausting family $\B=(B_R)_{R>0}$ of relatively compact sets $B_R\subset X$.
\end{itemize}
For $R>0$ we then set 
$$N_R(D, X):=\# \{ d\in D \mid d\ \in B_R\}\, .$$
For a measurable subset $B\subset X$ we use  the notation $\vol(B) =|B|=\mu(B)$.  
We then ask to what extent $N_R(D,X)$ approximates $\vol(B_R)$  for $R\to \infty$.   We say that the
quadruple 
$(X,\mu, D, \B)$  satisfies {\it main term counting} (MTC) provided that 
 
\begin{equation} \lim_{R\to \infty}    {N_R(D, X) \over |B_R|}  = 1\, .\end{equation}

The mother of all counting problems is the {\it Gau\ss{} circle problem} (GCP), that is 
$(\R^2, dx \wedge dy, \Z^2, \B)$ with $B_R=B_R^{\rm Eucl} $ the round Euclidean ball of radius $R$.  
It is almost immediate that the GCP satisfies MTC. 

\par In order to expect MTC in a general setup explained above one needs additional assumptions. In some sense 
the discrete set needs to be equidistributed at infinity.  This might be satisfied if $D$ is {\it freely  homogeneous}, that is:
\begin{itemize}
\item There is an infinite discrete group $\Gamma$ acting on $X$ properly,  freely and volume-preserving. 
\item There is $x_0\in X$ such that $\Gamma\simeq \Gamma\cdot x_0 =D$. 
\item There is a (locally closed) fundamental domain $F\subset X$ for the 
$\Gamma$-action with $\vol(X/\Gamma):=\vol (F) =1$ and $x_0\in F$.
\end{itemize}

At least for $F$ relatively compact and the family $(B_R)_{R>0}$ exhausting the space  
in a homogeneous manner, we 
can imagine that $\vol (B_R)$  is asymptotically approximated by 
$$\#\{ \gamma\in \Gamma  \mid \gamma \cdot F \subset B_R\}, $$
i.e. the number of tiles $\gamma\cdot F$ which lie in $B_R$.    In fact, it is easy to construct a family $\B$ which satisfies 
MTC. For that  let $(F_R)_{R>0}$ be a relatively compact exhaustion of the fundamental domain $F$ and 
$(\Gamma_R)_{R>0}$ an exhaustion of $\Gamma$ by finite subsets.  Then $B_R:= \Gamma_R \cdot F_R$ defines 
an exhaustion of $X$ which satisfies MTC. 
In practice we certainly wish to take more geometric exhaustions $\B$ than the one constructed above.  Typically 
has a variety of interesting metrics $d$ on $X$ and one would like to take for $\B$ the metric balls
$B_R:=\{ x\in X\mid  d(x,x_0)<R\}$.  We return to this issue later on.

\par Here is a large class of freely homogeneous examples where MTC holds.  We let $X=G$ be connected Lie group 
which admits a lattice $\Gamma<G$. We recall that a lattice in a Lie group $G$ is a discrete subgroup $\Gamma<G$ 
with finite co-volume with respect to a Haar measure $\mu$ on $G$.  We take $D=\Gamma$ and normalize $\mu$ such that 
$\vol (G/\Gamma)=1$. 
For the exhausting family $\B $ almost anything will do; a particular nice family would be  
balls of radius $R$ with respect to a left invariant metric on $X=G$.  The quickest way to establish 
MTC for $G$  is via the wavefront lemma applied to $G$ viewed as a homogeneous
space for the two-sided action of $G\times G$, see 
\cite{EM} and Remark \ref{WFgpcase} after Lemma \ref{wfl} below.
A further study yields also an error estimate for this case, see Theorem 1.5 of \cite{GNbook}.

\par Starting with a freely homogeneous quadruple 
$(X, \mu, \Gamma, \B)$ which satisfies MTC we let 
$p: X\to Z:= H\bs X$ be a principal fiber bundle with 
fibre $H$. We assume that $H$ is  a Lie group.  Let $z_0=p(x_0)$ and identify 
$H$ with $H \cdot x_0= p^{-1}(z_0)$. 
Let $\mu_H$ be a  right Haar-measure on $H$ and $\mu_X=\mu$. We request that there is a Radon measure $\mu_Z$ on $Z$ such that 

\begin{equation}\label{measure compatible}  \int_X  f(x) \ d\mu_X(x) = \int_Z  \int_H  f (h \tilde z)  \ d\mu_H(h) \ d\mu_Z(z)\qquad (f\in C_c(X))\end{equation} 
where $Z \to X,  \ z\mapsto \tilde z$ is some measurable cross section.

The following assumptions on the fibre are then natural: 
\begin{itemize} 
\item There is a discrete subgroup $\Gamma_H <\Gamma$ such that $H \cap (\Gamma \cdot x_0) = \Gamma_H \cdot x_0$.
\item $\Gamma_H$ is a co-volume $1$ lattice in $H$. 
\item $(H, \mu_H, \Gamma_H,  \B^H)$   satisfies MTC where $\B^H:= (B_R \cap H)_{R>0}$. 
\end{itemize}

Let now $B_R^Z:=p(B_R)$ and $\B^Z$ the corresponding family of balls.  One might then ask whether   $p(\Gamma)\subset Z$ is discrete and 
$(Z, \mu_Z,  p(\Gamma),  \B^Z)$ satisfies MTC? 

In case the principal bundle is homogeneous we have the following:

\begin{con} \label{p1} Let  $G$ be a connected Lie group and $H<G$ a closed subgroup such that $Z:=G/H$ carries 
a $G$-invariant positive Radon measure $\mu_Z$ which satisfies  (\ref{measure compatible}) with respect to some 
Haar measures $\mu_G$ and $\mu_H$ of $G$ and $H$.  Assume that $H$ has finitely many connected components. 
Further let $\Gamma<G$ be a lattice such that 
\begin{itemize}
\item  $\vol (G/\Gamma)=1$.
\item  $\Gamma_H:=\Gamma\cap H$ is a lattice in $H$ such that $\vol(H/\Gamma_H)=1$. 
\end{itemize}
Then there exists an exhausting compact family $\B^Z=(B_R^Z)_{R>0}$ of $Z$ such that 
the quadruple $(Z,\mu_Z, \Gamma/\Gamma_H, \B^Z)$ satisfies MTC. 
\end{con}

We point out here that the setup in the above open problem was taken from \cite{DRS}.  
Notice that the GCP falls in this setup 
as well as the {\it Selberg circle problem} (SCP)   on the upper half plane $Z=\Sl(2,\R)/ \SO(2,\R)$ on which we will comment 
in more detail later on.

\begin{rmk} {\rm (Heuristics for Conjecture \ref{p1})} We now give some heuristics for Conjecture 
\ref{p1} in case $Y_H:= H/\Gamma_H$ is compact.
Let us assume first that the fibration $G \to Z=G/H$ is trivial, that is 
$G\simeq Z \times H$ as a right $H$-space. This is already an interesting class and typical examples arise in the following manner: We let $G$ be semisimple, 
$G=KAN$ be an Iwasawa decomposition and $N=N_1 \times N_2$.  Then for $H=N_2$, the principal bundle $G\to Z$ 
is trivial, as $Z\simeq K \times A \times N_1$ as manifolds.  
\par To continue, we assume that $B_R^G = B_R^Z \times B_R^H$ 
are product balls.  This yields in particular
\begin{equation} \label{volume factor} \vol (B_R^G)  = \vol (B_R^Z) \cdot \vol(B_R^H)\, .\end{equation} 
As we assume that  $Y_H:= H/\Gamma_H$ is compact we have  
$H= B_R^H \Gamma_H$ for sufficiently large $R$, say $R\geq R_0$.   We need to be a bit more specific on the balls
$B_R^H$ and request that they are of the form $B_R^H=F_H \Gamma_{H,R}$ with $\1_H\in F_H\subset H$ a fundamental 
domain for the right action of $\Gamma_H$ on $H$ and $\1_H\in  \Gamma_{H,R}\subset \Gamma_H$ an exhausting family 
of finite subsets.  We claim that 

\begin{equation} \label{factor count} N_R(\Gamma, G) = N_R(\Gamma_H, H)  \cdot N _R (\Gamma/\Gamma_H, Z)  \qquad (R\geq R_0)\, .\end{equation} 
In fact, let $\tilde N_R(\Gamma,G)=\{ \gamma\in \Gamma\mid \gamma\in B_R^G\}$ so that $\# \tilde N_R(\Gamma, G)= N_R(\Gamma, G)$. Let $\gamma\in \tilde N_R(\Gamma, G)$ and $\oline \gamma= \gamma H$ be its image in $Z$.  Clearly $\oline \gamma
\in B_R^Z$.  Conversely, let $\oline \gamma=\gamma H \in B_R^Z$ with $\gamma\in \Gamma$.  Then 
$\gamma\in B_R^Z H = B_R^Z  F_H \Gamma_H$.  This allows us to assume that $\gamma\in B_R^Z F_H \subset B_R^Z B_R^H$, i.e.
$\gamma\in \tilde N_R(\Gamma, G)$.  This shows in particular that the map 
$$\tilde N_R(\Gamma, G) \to \tilde N_R(\Gamma/ \Gamma_H, B_R^Z), \ \gamma \mapsto \oline \gamma$$
is surjective with fibers isomorphic to $\Gamma_{H,R}= \tilde N_R(\Gamma_H, B_R^H)$. 
We know that MTC holds for groups and all reasonable balls. 
If we assume now that MTC holds for the above class of balls $B_R^G$ in $G$ (which in the example of $H=N_2$ abelian unipotent can be easily verified), then we deduce from (\ref{volume factor}) and (\ref{factor count}) that MTC holds in case $G\simeq Z \times H$ is trivial.

\par To adapt the argument to the case where $G\to Z$ is not trivial we exhibit now 
an open and measure dense subset $Z'\subset Z$ for which the bundle trivializes and which avoids
all lattice points.  To obtain $Z'$ we first find a countable cover of open sets $U_1,U_2,\ldots$ of 
$Z=G/H$ such that $G\to Z$ trivializes over each $\oline U_i$. We may assume that each $\oline U_i$
is diffeomorphic to a closed euclidean ball.  Now consider the disjoint union 

$$Z':=U_1\cup U_2 \bs \oline U_1\cup U_3 \bs (\oline U_1\cup \oline U_2)\cup \ldots \,. $$
Clearly $Z'$ is open, measure dense in $Z$ and $G\to Z$ trivializes over $Z'$.  Moreover as
$D:=\Gamma/\Gamma_H\subset Z$ is countable we can arrange that it avoids $\bigcup_{j=1}^\infty \partial U_j$, i.e.
$D\subset Z'$. Now proceed as above.
\end{rmk} 

We keep the setup of Conjecture \ref{p1}. A particular nice class of balls can be defined by the Mostow-decomposition of $G/H$ which we 
now recall. \par Let $K<G$ be a maximal 
compact subgroup of $G$ such that $K_H:=K\cap H$ is a maximal compact subgroup of $H$.  Let 
$N_K(H)$ be the normalizer of $H$ in $K$. 
The Mostow decomposition (see Theorem \ref{Mostow dec} in the appendix)  
of $Z=G/H$ then asserts the existence  of a finite dimensional vector space 
and  $N_K(H)$-module $V\subset \gf$ such that 
\begin{equation} \label{Mostow1}  K \times_{K_H} V \to G/H,  \ \ [k,X]\mapsto k \exp(X)H\end{equation}
is a diffeomorphism.

\par The coordinates given by (\ref{Mostow1}) allow us to define 
a natural family of balls.   Let  $\|\cdot\|$ be a $N_K(H) $-invariant Euclidean norm on $V$ and define $K$-invariant 
balls as follows
$$B_R:=\{ [k,X]\in G/H \mid   \|X\|<R\}\, .$$  
Following \cite{KSS} we call these balls {\it intrinsic}.  Having this terminology we ask: 

\begin{problem} Does  Conjecture \ref{p1} hold true with $\B^Z$ a  family of intrinsic balls ?
\end{problem}

\bigskip  From now on we let $G$ be a real reductive group and $H<G$ be a closed subgroup with finitely many 
connected components and real algebraic Lie algebra.  Further we let $\Gamma<G$ be a lattice 
in the setup of Conjecture \ref{p1}.  Further we set $Y:=G/\Gamma$. 

\par Here is a short history of MTC in the context of Conjecture \ref{p1}. 
MTC was established via harmonic analysis in \cite{DRS}
for symmetric spaces $G/H$ and 
certain families of balls, for lattices with $Y_H$ compact.  In subsequent work \cite{EM} the obstruction that $Y_H$ is compact was removed 
and an ergodic
approach was presented.  The ergodic techniques were refined in  \cite{EMS} and it was 
discovered that MTC holds for a
wider class of reductive spaces: For reductive algebraic groups $G, H$ 
defined over $\Q$ and arithmetic lattices $\Gamma<G(\Q)$  it is sufficient to request that 
the identity component of $H$ is not contained in a proper parabolic subgroup of $G$ 
which is defined over $\Q$ 
and that the balls $B_R$ satisfy a certain condition of {\it non-focusing}.

\par In these works the balls $B_R$ are constructed as follows. 
All spaces considered are affine in the sense that 
there exists a $G$-equivariant embedding of $Z$ into the representation module 
$V$ of a rational representation of $G$. For any such embedding and any norm on the vector 
space $V,$ one then obtains a family of balls $B_R$ on $Z$ by intersection with the metric balls 
in $V$. For symmetric spaces all families of balls produced this way
are suitable for the lattice counting, but in general one needs to 
assume non-focusing in addition.

\par The core of the approach of \cite{EM} was a geometric lemma satisfied by symmetric spaces 
which the authors termed {\it wavefront lemma}.  In their work on p-adic spherical spaces Sakellaridis and Venkatesh \cite{SV}
coined the notion of a wavefront p-adic spherical space and showed that these spaces satisfy a p-adic   version of 
the wavefront lemma.  Wavefront real spherical spaces were introduced in \cite{KKSS1} and it was shown 
in \cite{KKSS1} Lemma 6.3   that 
they satisfy the wavefront lemma  of \cite{EM}.  In particular real spherical spaces of wavefront 
type satisfy MTC for all reasonable families of balls
(see Theorem \ref{MTC wf} below).

\par At this point we remark that all symmetric spaces are real spherical of wavefront type.   
The latter type of spaces 
is going to be the main player of this article.  Before we continue with error term bounds for the lattice counting problem
we insert a section on wavefront spaces and provide a proof of the wavefront lemma.

\section{Real spherical spaces}\label{wfrss}

\par The notational convention for this paper is that we denote Lie groups by upper case Latin letters, e.g. 
$A, B, C$, and their corresponding Lie algebras with lower case German letters, e.g.
$\af, \mathfrak{b}, \cf$.

\par We assume that $Z=G/H$ is {\it real spherical}, that is, a minimal 
parabolic subgroup $P$ of $G$ admits an open orbit on it. By choosing
$P$ suitably we can then arrange that its orbit through the origin $z_0=H\in Z$ is open, 
or equivalently that $\gf=\hf +\pf$. All symmetric spaces are known to be real spherical.

\par According to \cite{KKS} there is a unique parabolic subgroup $Q\supset P$ 
with the 
following two properties: 
\begin{itemize}
\item $QH=PH$.
\item There is a Levi decomposition $Q=LU$ with $L_{\rm n} \subset Q\cap H\subset L$. 
\end{itemize}
Here $L_{\rm n}$ denotes the analytic subgroup of $L$ for which the Lie algebra
$\lf_{\rm n}$ is the sum of all non-compact simple ideals of $\lf$. 

We observe that $L\cap P$ is a minimal parabolic subgroup in $L$. 
It follows that we can choose an Iwasawa decomposition
$L=K_L A_L N_L$ such that $A_LN_L\subset P$. Having fixed 
that we choose a compatible Iwasawa decomposition $G=KAN$, i.e.
$K_L<K, A_L =A$ and $N_L<N$. Then $N$ is the unipotent radical of $P$ and 
with $M=Z_K(A)$ we have the Langlands decomposition $P=MAN$ of $P$.

\par Set $A_H:=A\cap H$ and put $A_Z=A/A_H$. We recall that $\dim A_Z$ is an 
invariant of the real spherical space, called the real rank (see \cite{KKS}) and denoted by 
$\rank_\R Z$.

Attached to $\af$ and $P$
are the root system $\Sigma=\Sigma(\af, \gf)\subset \af^*\bs\{0\}$ 
and its set $\Sigma^+$ of positive roots. For the
associated root space decomposition
$\gf = \oplus_{\alpha\in \{0\}\cup\Sigma}  \gf^\alpha$
we have  $\gf^0=\af\oplus\mf$ and $\nf=\oplus_{\alpha\in \Sigma^+}  \gf^\alpha$.
We write $\Sigma_\uf\subset\Sigma^+$ for the subset of $\af$-weights of $\uf$, and
let $\oline\nf$ and $\oline\uf$ denote the corresponding sum of negative root spaces.
Then 
$$\gf = \oline\uf \oplus \lf\oplus \uf=
\bigoplus_{\alpha\in\Sigma_\uf}  \gf^{-\alpha} \oplus \lf \oplus
\bigoplus_{\beta\in\Sigma_\uf}  \gf^\beta$$

Attached to $Z$ is a geometric invariant, the so-called compression cone.
It is a closed and convex subcone $\af_Z^-$ of $\af_Z$, defined as follows.
According to \cite{KKS} there exists a linear map
\begin{equation}\label{Tmap}
T: \oline \uf \,\to\,   
\lf_H^{\perp}\oplus \uf\subset \bigoplus_{\beta\in \{0\}\cup\Sigma_\uf} \gf^{\beta} 
\end{equation}
such that
\begin{equation}\label{h as graph}
\hf = \lf \cap \hf \oplus \{ \oline X + T(\oline X)\mid \oline X \in \oline\uf\}.
\end{equation}
Here  $\lf_H^{\perp}\subset \af\oplus\mf$ denotes the orthocomplement of $\lf \cap \hf$ in $\lf$.
For $\alpha\in \Sigma_\uf$ and $\beta\in \{0\}\cup\Sigma_\uf$ we denote by 
$T_{\alpha,\beta}: \gf^{-\alpha}\to \gf^\beta$
the map obtained from (\ref{Tmap}) by restriction of $T$ to
$\gf^{-\alpha}$ and projection to $\gf^\beta$.
Then $$T=\sum_{\alpha,\beta} T_{\alpha,\beta}\, .$$ 
For each $Y\in\af_H$ and all  $\oline X\in\oline\uf$ we find
$$\hf\ni [Y, \oline X+ T_{\alpha,\beta}(\oline X)]=
-\alpha(Y)\oline X+ \beta(Y)T_{\alpha,\beta}(\oline X).$$
Hence if $T_{\alpha,\beta}\neq 0$ we see by comparing with
(\ref{h as graph}) that $-\alpha(Y)=\beta(Y)$. Hence 
$${\mathcal M}:=\{\alpha +\beta\mid T_{\alpha, \beta}\neq 0\}$$ 
can be viewed as a subset of $\af_H^\perp=\af_Z^*$.  According to \cite{KK}, Cor.~12.5, the cone 
$C_{\mathcal M}\subset\af_Z^*$ generated by ${\mathcal M}$ is simplicial and  
contains a linearly independent set $S$ of generators such that
${\mathcal M} \subset \N_0[S]$.  Such a set $S$ will be refered to as 
{\it a set of spherical roots for $Z$.} We fix any such set and define
$$\af_Z^-:=\{ Y\in \af\mid  (\forall \sigma\in S) \ \sigma(Y)\leq 0\} \, .$$

\subsection{ Wavefront spaces}
In \cite{KKSS1}, Section 6, we defined the notion of {\it wavefront} for 
real spherical spaces, which will now be recalled. 

If we denote by  $\af^{-}\subset\af $ the closure of the 
negative Weyl chamber, then it is clear that
$\af^-\subset\af_Z^-$. Hence 
\begin{equation*}   
\af^-+\af_H \subset \af_Z^-.
\end{equation*}
Since $T_{\alpha,\beta}$ can be zero for many pairs
of roots $(\alpha,\beta)$, the above inclusion can be proper in general.
By definition
$Z$ is called a {\it wavefront space} if in fact
\begin{equation} \label{wf}  \af^-+\af_H= \af_Z^- .\end{equation}

We shall now give some examples of wavefront spaces.
We first note that if $Z$ is symmetric, say with corresponding
involution $\sigma$ of $\gf$, then the special Iwasawa decomposition
which we requested above can be obtained by choosing a Cartan involution 
$\theta$ that commutes with $\sigma$. Then
$T(\oline X)=\sigma(\oline X)$ for all $\oline X\in\oline u$,
and $T_{\alpha,\beta}$ is non-zero if and only if $\beta=-\sigma\alpha$
in this case. From this it easily follows that we have the equality
in (\ref{wf}), that is, {\it all symmetric spaces are wavefront}.

Next we give a classification of all non-symmetric
wavefront real spherical pairs $(\gf,\hf)$ with $\gf$ simple and $\hf$ reductive. 
These are the pairs of the following Table 1. 
The table is deduced from the classification in \cite{KKPS} together with 
the following result from \cite{KKSS2} Thm.~6.3.

\begin{lemma}\label{lemma compact quotient}
Let $Z=G/H$ be a wavefront real spherical space with $\gf$ simple and $\hf$ reductive. 
Let $H<H^\star<G$ be a closed subgroup such that  
$Z^\star:= G/H^\star$ is unimodular. Then $H^\star/H$ is compact. 
\end{lemma}

\begin{table}[htbp]
$$\begin{array}{rllll}
&\gf&\hf\\
\hline
\yy&\su(p_1+p_2,q_1+q_2)&\aas\su(p_1,q_1)+\su(p_2,q_2)&&(p_1,q_1)\ne(q_2,p_2)\\
\yy&\su(n,1)&\su(n-2q,1){+}\sp(q)+\ff&\ff\subseteq\uf(1)&1\le q\le\frac n2\\
\yy&\sl(n,\HH)&\sl(n-1,\HH)+\ff&\R \subseteq\ff\subseteq\C&n\ge3   \\
\yy&\sl(n,\HH)&\aas\sl(n,\C)&&n\text{ odd}\  \\
\hline
\yy&\sp(p,q)&\aas\su(p,q)&&p\ne q\\
\yy&\sp(p,q)&\aas\sp(p-1,q)&&p\ge1\\
\hline
\yy&\so(2p,2q)&\aas\su(p,q)&&p\ne q\\
\yy&\so(2p+1,2q)&\aas\su(p,q)&&p\ne q-1,q\\
\yy&\so(n,1)&\so(n-2q,1)+\su(q)+\ff&\ff\subseteq\uf(1)&2\le q\le\frac n2\\
\yy&\so(n,1)&\so(n-4q,1)+\sp(q)+\ff&\ff\subseteq\sp(1)&2\le q\le\frac n4\\
\yy&\so(n,1)&\so(n-16,1)+\spin(9)&&n\ge16\\
\yy&\so(n,q)&\so(n-7,q)+\sG_2&&n\ge7, q=1,2\\

\yy&\so(n,q)&\so(n-8,q)+\spin(7)&&n\ge8, q=1,2,3\\
\yy&\so^*(2n)&\aas\so^*(2n-2)&&n\ge5\\
\yy&\so^*(10)&\aas\spin(6,1)\text{ or }\ \ \aas\spin(5,2)\\
\yy&\so(4,3) & \sG_2^1 \\ 
\yy & \so (7,\C) & \sG_2^\C\\ 
\hline
\yy&\sE_6^4&\sl(3,\HH)+\ff&\ff\subseteq\uf(1)\\
\yy&\sE_7^2&\aas\sE_6^2\text{ or }\ \  \aas\sE_6^3\\
\yy&\sF_4^2&\sp(2,1)+\ff&\ff\subseteq\uf(1)\\
\hline
\yy&\sG_2^\C&\sl(3,\C) & \\
\yy&\sG_2^1&\sl(3,\R), \su(2,1) & \\
\end{array}$$
\smallskip
\centerline{\rm\Tabelle{table real spherical}}

Cases 
marked \ $\aas$ result from a symmetric over-algebra $\hf^* \supset \hf$  
such that $\hf +\uf(1)=\hf^*$. 
\end{table} 

\begin{rmk}\label{RR1} We say that $Z$ has real rank one if $\dim \af_Z=1$.  In Table \ref{table real spherical}
the following are of real rank one: 
$$ \hbox{(2), (6), (9) - (11), (12) and (13) for $q=1$, (20) - (22) }\, .$$ 
\end{rmk}

There are many more examples in case $\gf$ is semi-simple and not simple  (see the classification in \cite{KKPS2}). 
However, with the exception of two cases, they are not interesting for the lattice count for the following reason: 
If $G=G_1\times \ldots\times G_n$ is a product of simple groups, then an irreducible lattice can only exist 
if the $\gf_i\otimes_\R\C$ are all isomorphic (see \cite{J}).  
In view of the classification in \cite{KKPS2} one is then left with the group case 
$(\gf, \hf) = (\gf_0\oplus \gf_0, \diag(\gf_0))$ 
and the triple spaces
$$\gf = \so(1,n)\oplus \so(1,n) \oplus\so(1,n) \quad\hbox{and}  \quad \hf=\diag \so(1,n) \qquad  (n\geq 2)\, .$$
The triple spaces feature a lot of interesting irreducible lattices. Here we review an example given in 
\cite{KSS}.  

We let $n=2$ and $G_0=\SO_e(1,2)$ and consider space $Z= G/H$ where 
$G=G_0\times G_0\times G_0$ and $H = \diag(G_0)$. 

We assume that $G_0=\SO_e(1,2)$ is defined by a quadratic form $Q$ which has 
integer coefficients and is anisotropic over $\Q$, for example 
$$Q(x_0, x_1, x_2)= 2x_0^2 - 3x_1^2 - x_2^2\, .$$
Then, according to Borel,   $\Gamma_0=G_0(\Z)$ is a uniform lattice in 
$G_0$. 

\par Next let $k$ be a cubic Galois extension of $\Q$. 
Note that $k$ is totally real. An example of $k$ is the splitting field
of the polynomial $f(x)=  x^3 + x^2 - 2x - 1$.  Let $\sigma$ be a generator of 
the Galois group of $k|\Q$.  Let ${\mathcal O}_k$ be the ring of algebraic 
integers of $k$.  We define $\Gamma<G=G_0^3$ to be the image 
of $G_0({\mathcal O}_k)$ under the embedding 

$$ G_0({\mathcal O}_k)\ni \gamma \mapsto (\gamma, \gamma^\sigma, 
\gamma^{\sigma^2})\in G\, .$$
Then $\Gamma<G$ is a uniform irreducible lattice with $H\cap \Gamma\simeq \Gamma_0$
a uniform lattice in $H\simeq G_0$.

\subsection{The polar decomposition and integral inequalities} 
In this section  $Z=G/H$ is a real spherical space. It is not assumed to be wavefront
unless requested otherwise.

Let us denote by $z_0=H\in Z$ the standard base point.  It will be 
convenient to 
assume that there are complex groups 
$H_\C \subset G_\C$ such that $G\subset G_\C$ is a real form and $H=G\cap H_\C$.  
Set $Z_\C=G_\C/H_\C$  and observe that 
\begin{equation}\label{Z in Z_C} 
Z\hookrightarrow Z_\C, \ \ gH\mapsto gH_\C
\end{equation} 
constitutes a $G$-equivariant embedding.  

We now recall from \cite{KKSS1} (see also \cite{KK}, Sect. 13)  the polar decomposition 
\begin{equation}\label{polar} Z= \Omega A_Z^- \W\cdot z_0\end{equation}
where 
\begin{itemize}
\item $\Omega$ is a compact set of the form $\Omega=FK $ with $F\subset G$ a finite set. 
\item $\W\subset G$ is a finite set with the property that 
 $P\sw\cdot z_0\subset Z$ is open for all $\sw\in\W$.
Moreover, when (\ref{Z in Z_C}) is assumed it can be arranged that
$\W\cdot z_0 \subset T \cdot z_0 \cap Z$
where $T=\exp(i\af)$ and the intersection is taken in $Z_\C$. 
\end{itemize} 

We recall the definition of the volume weight $\vb$ from Section \ref{Sect1} and record the basic estimate 
from \cite{KKSS2}: Let $\rho=\frac12\operatorname{tr}\ad_\uf\in\af^*$.
Assume that $Z$ is unimodular. Then $\rho|_{\af_H}=0$  (cf. Lemma 4.2 in \cite{KKSS2})  and hence $\rho$ descends to a functional on $\af_Z$. Moreover, there exists constants $C_1, C_2>0$ such that 
\begin{equation}\label{vbound}   C_1 a^{-2\rho} \leq\vb(\omega a \sw\cdot z_0)\leq 
C_2 a^{-2\rho} \qquad (\omega\in \Omega, a\in A_Z^-, \sw\in \W)\, .\end{equation}

In the sequel it will be convenient to realize $\af_Z$ via $\af_H^\perp\subset \af$.  We denote by 
$\af_Z^{--}$ the interior of $\af_Z^-$.

 \begin{lemma} \label{lemma int-ineq}{\rm (Integral inequalities)} Let $Z$ be a unimodular real spherical space. Then the following integral 
inequalities hold: 
\begin{enumerate} 
\item \label{intineq2} There exist an element $a_1 \in A_Z^{--}$ such that for all 
$a_0\in a_1 A_Z^-$ the following assertions hold: After replacing $K$ by $\Ad(a_0)^{-1} K$ 
there exists constants $c,C>0$  only depending on the normalization of measures, such that
for all measurable function $f: Z\to \R_{\geq 0}$ with $\supp f \subset K A_Z^- \W\cdot z_0$ 
one has
$$ c \int_Z f(z) \ dz \leq \sum_{\sw \in \W} \int_{K}\int_{A_Z^-}  f (ka\sw\cdot z_0) a^{-2\rho}  \ da \ dk \leq C \int_Z f(z) \ dz $$
where $dz, dk, da$ are Haar measures on $Z, K, A_Z^-$. 

\item\label{intineq1} Assume that  $Z$ is wavefront type. 
Then there exists a compact neighborhood $B$ of $\1$ in $G$ such that 
$Z= B A_Z^-\W\cdot z_0$ and a constant $C>0$, only depending on $B$ 
and the normalization of measures,  such that for all measurable functions 
$f: Z \to \R_{\geq 0}$ one has 
$$\int_Z  f(z)\ dz \leq   C \sum_{\sw\in \W} \int_B \int_{A_Z^-}  f(ga\sw\cdot z_0)  a^{-2\rho}  \ da \ dg $$
with $dg$ a Haar measure on $G$. 
\end{enumerate} 
\end{lemma} 

The proof of Lemma \ref{lemma int-ineq} (\ref{intineq1}) rests on the following geometric lemma which is 
of independent interest.

\begin{lemma} \label{volumes project}  {\rm (Volumes of balls under projection)} Assume that 
$Z=G/H$ is a unimodular real spherical space of wavefront type.  Fix  a compact neighborhood $B_1$ of $\1$ in $G$. 
For $E\subset B_1$ a measurable subset and $a\in A_Z^-$ 
we let $E_a \subset B_1^2= B_1\cdot B_1$ be the inverse image of $Ea\cdot z_0$ under the quotient 
morphism 
$$ B_1^2 \to  B_1^2 a\cdot z_0, \ \ b\mapsto ba \cdot z_0\, .$$
Then there exists a constant $C=C(B_1)>0$ such that 

\begin{equation}\label{claim} \vol_Z(Ea\cdot z_0) \leq C a^{-2\rho}  \vol_G(E_a ) \qquad (a \in A_Z^-, E\subset B_1\ \text{measurable})\, .\end{equation}
\end{lemma}

\begin{proof} Since $E B_1 \subset B_1^2$ we first record that

\begin{equation}\label{B_0^a}  E_a \supset E (B_1 \cap \Ad(a)H) =  Ea( \Ad(a)^{-1} B_1 \cap H)a^{-1}\, .\end{equation}

\par If $f$ is a compactly supported measurable function on $G$ we denote by $f^H$ its $H$-average:
$f^H(gH)= \int_H f(gh) \ dh$.  The claim is then a consequence of 
\begin{equation} \label{claim modi} \1_{Ea\cdot z_0} \leq  C a^{-2\rho}\cdot  \1_{E_a  a}^H\end{equation} 
as functions on $Z$, uniformly for all $E\subset B_1$ and all $a\in A_Z^-$. 
In fact, integrating (\ref{claim modi}) over $Z$ yields  (\ref{claim}).
Now, for $b\in E$  we record 

\begin{eqnarray*}  \1_{E_aa}^H (b a\cdot z_0) &=&\vol_H\{ h\in H\mid bah \in  E_a a\}\\
&\geq& \vol_H \{ h\in H \mid  ba h\in E a ( \Ad(a)^{-1}B_1 \cap H) \}\\ 
&\geq & \vol_H( (\Ad(a)^{-1}B_1) \cap H)\end{eqnarray*}
where we used (\ref{B_0^a}) for the first inequality.  Thus the lemma will follow 
from 
\begin{equation} \label{vol-claim}\vol_H( (\Ad(a)^{-1}B_1) \cap H) \geq  C a^{2\rho}\qquad (a \in A_Z^-)\end{equation} 
for a constant $C$ only depending on $B_1$.  To see that let $B_U$, $B_{\oline U}$, $B_L$ be 
convex compact neighborhoods of $\1$ in $U, \oline U, L$ such that $B_1 \supset B_U B_L B_{\oline U}$. 
Because $Z$ is assumed to be wavefront, we note that $\Ad(a)^{-1}B_U \supset B_U$,  $\Ad(a)^{-1} B_{\oline U} \subset B_{\oline U}$ 
and $\Ad(a)^{-1}B_L =B_L$. 

Thus, if we replace $B_1$ by $B_U B_L B_{\oline U}$ we get 
\begin{equation}\label{balls compressed}  \Ad(a)^{-1} B_1 \supset B_U B_L (\Ad(a)^{-1}B_{\oline U})\, . \end{equation}

Recall now the vector decomposition $\hf= \lf\cap \hf + {\mathcal G}(T)$ with 
${\mathcal G}(T)\subset \oline \uf + (\uf +  \lf_H^\perp)$ the graph of the linear map 
$T$, see (\ref{Tmap}).  Let now $Y_1, \ldots, Y_d$ be  basis of $\lf\cap \hf$  and $U_1, \ldots, U_l$
be a basis of ${\mathcal G}(T)$ with each $U_i$ of the form $U_i = X_{-\alpha_i} + T(X_{ -\alpha_i})$
for some $X_{-\alpha_i} \in \gf^{-\alpha_i} \subset \oline \uf$. Then it follows from (\ref{balls compressed})
that there exists an $\e>0$ such that 

$$ \exp \left(  \prod_{i=1}^d (-\e,\e)Y_j \times  \prod_{j=1}^l  a^{\alpha_i}  (-\e, \e) U_i\right)\subset
(\Ad(a)^{-1}B_1) \cap H\qquad (a\in A_Z^-)\,.$$
This establishes (\ref{vol-claim}) and hence the lemma.
\end{proof}

\begin{proof}[Proof of Lemma \ref{lemma int-ineq}] We first prove (\ref{intineq1}).  The arguments which follow will show that 
it is no restriction to assume that $\W=\{\1\}$.  So let us assume that in order to save 
notation. 
\par We start with a compact neighborhood $B_1$ of $\1$ in $G$ such that $B_1A_Z^-\cdot z_0=Z$. Let $E\subset B_1$ be a measurable subset.  For $a\in A_Z^-$ 
we let $E_a \subset B_1^2= B_1\cdot B_1$ be as in Lemma \ref{volumes project} above. 

\par Next let $B_{A_Z}$ be a fixed compact neighborhood of $\1$ in $A_Z$.  We assume that 
$B_{A_Z}=B_{A_Z}^{-1}\subset B_1$. 
We wish to estimate the integral 
$$I_1(a, E):=\int_{B_1^3}  \int_{a B_{A_Z}}    \1_{Ea\cdot z_0} (ba' ) (a')^{-2\rho} \ da'\ db\, .$$
First note that there exists a constant $C$ depending on the size of $B_{A_Z}$ such that 
$$I_1(a,E)\geq I_2(a, E):=C a^{-2\rho}   \int_{B_1^3} \int_{B_{A_Z}}  \1_{E a\cdot z_0}  (ba'a)  \ da' \ db\, .$$
Fubini then gives 
\begin{equation}\label{int2a}  I_2(a,E)= 
C a^{-2\rho}\int_{B_{A_Z}} \int_{B_1^3}   \1_{E a\cdot z_0}  (ba'a)  \ db \ da'\, .\end{equation} 
Now observe for $a'\in B_{A_Z}$

$$\{ b \in B_1^3\mid   ba'a \cdot z_0\in E a\cdot z_0\} \supset \{ b \in B_1^2 (a')^{-1}\mid  b a' a\cdot
z_0\in E a\cdot z_0\} = E_a  (a')^{-1}\, .$$
Hence we get from (\ref{int2a}) and (\ref{claim}) that 
\begin{eqnarray*}  I_2(a,E)
&\geq&C a^{-2\rho}  \int_{B_{A_Z}} \vol_G(E_a) \ da'\\
&=&C a^{-2\rho}   \vol_{A_Z} (B_{A_Z}) \vol_G(E_a) \\
&\geq&C \vol_Z(Ea\cdot z_0)\, .\end{eqnarray*}

In particular we have shown that 
$$\int_Z  \1_{Ea\cdot z_0} (z) \ dz \leq C \int_{B_1^3}  \int_{B_{A_Z} A_Z^-}  \1_{Ea\cdot z_0} (ba')  (a')^{-2\rho}
\ da' \ db\, .$$
Since $E\subset B_1$ was arbitary and $B_1 A_Z^-\cdot z_0=Z$ we thus get for all 
measurable functions $f\geq 0$ on $Z$ that 
$$\int_Z  f(z) \ dz \leq C \int_{B_1^3}  \int_{B_{A_Z} A_Z^-}  f(ba\cdot z_0)  a^{-2\rho}
\ da\  db\, .$$
This implies with $B:=B_1^4$ that 
$$\int_Z  f(z) \ dz \leq C \int_{B}  \int_{A_Z^-}  f(ba\cdot z_0)  a^{-2\rho}
\ da\  db$$
and completes the proof for (\ref{intineq1}) -- notice that the flow of arguments 
showed we can squeeze elements $\sw\in\W$ between $a\cdot z_0$ without doing any harm.

\par We move to the proof of (\ref{intineq2}) which is implicitly contained in the proof of 
\cite{KKS2}, Th.~8.5.  We summarize the argument, assuming for simplicity the existence of (\ref{Z in Z_C}).

\par Fix $\sw\in \W$, let $M_H:= M \cap H$ and consider the smooth map 
$$ \Psi_\sw:  K/ M_H  \times A_Z^-\to Z, \ \ (kM_H, a)\mapsto k a\sw\cdot z_0\, .$$
If we assume that $K$ is chosen such that $\kf +\af + \Ad(a\sw)\hf=\gf$ for all $a\in A_Z^-$ (which can be achieved 
by replacing $K$ with $\tilde a_0^{-1}K\tilde a_0$ for an appropriate $\tilde a_0\in A_Z^{--}\subset A$ (see (6.4) in \cite{KKS2}), then $\Psi_\sw$ has 
everywhere invertible differential, i.e. $\Psi_\sw$ is \'etale.  Moreover, as $\Psi_\sw$ arises from an algebraic map defined over $\R$, it is finite to one. Hence if we pull the invariant differential form on $Z$ back to $K/M_H \times A_Z^-$ via $\Psi_\sw$ we obtain that  

\begin{equation*} \int_{K}\int_{A_Z^-}  f(ka\sw\cdot z_0) J_{\sw}(a)  \ da \ dk \leq \int_Z f(z) \ dz \qquad (f\geq 0, f\ \text{measurable} ))\end{equation*}
for a positive function $J_\sw$ on $A_Z^-$.   Moreover, if $\supp f \subset KA_Z^-\sw\cdot z_0$
we obtain a constant $c>0$ such that 
\begin{equation*} c \int f(z) \ dz \leq \int_{K}\int_{A_Z^-}  f(ka\sw\cdot z_0) J_{\sw}(a)  \ da \ dk \end{equation*}
holds as well. 

\par Up to a constant factor $J_\sw(a)$ equals $|\det d\Psi_\sw(\1 M_H, a)|$ for which 
a formula was provided in the proof of Prop. 4.5  of \cite{KKSS2}: 

$$J_\sw(a) = |  \Ad(a\sw){\bf X} \wedge {\bf Y}\wedge {\bf U}|$$
where ${\bf X} =X_1\wedge\ldots \wedge X_d$ for a basis $X_1, \ldots, X_d$ of $\hf$,  ${\bf Y}=Y_1\wedge\dots \wedge Y_r$ associated to a basis of $\af_Z$ (viewed as a subspace of $\af$) and finally 
${\bf U}=U_1\wedge \ldots\wedge U_k$  with $U_1, \ldots, U_k$ a basis of $\mf_H^{\perp_\kf}\subset \kf$. 
Now for any $X\in \af_Z^{--}$ we recall from Remark 2.2 in \cite{KKSS2} that 
$$\lim_{t\to \infty} e^{2\rho (X)} e^{t \ad X} \Ad(\sw) {\bf X}
= {\bf X}_{\lim{}}$$
with ${\bf X}_{\lim{}}$ corresponding to a basis of the maximal deformation $\hf_{\lim{}}=\lf \cap \hf +\oline{\uf}$
of $\hf$. In more quantitative terms we have (see (4.4) in \cite{KKSS2})  up to normalization
$$J_\sw(a)  = a^{-2\rho}(1 + R(a))\qquad (a\in A_Z^-)$$ 
with $|R(a)|\leq C \max_{\alpha\in S} a^\alpha$ with $S$ a set of spherical roots.  In particular, we obtain 
$$J_\sw(aa_1) \asymp a^{-2\rho} \qquad (a\in A_Z^-)$$
provided that $a_1\in A_Z^{--}$ lies deep enough in the Weyl chamber.  Finally notice that the $a_1$-shift in the argument 
of $J_\sw$ above corresponds to replacing $K$ by $\Ad(a_1)^{-1}K$. 
This proves (\ref{intineq2}).

\end{proof}

In the sequel we request that $K$ is chosen such that the conclusion of 
Lemma \ref{lemma int-ineq} (\ref{intineq2})
holds.
\smallskip \par {\it For the rest of this article we assume that $Z$ is a real spherical space of wavefront type.}

\subsection{Property I}

We briefly recall some results and notions from \cite{KKSS2} and \cite{KSS}.

\par Let $(\pi,\Hc_{\pi})$ be a unitary irreducible
representation of $G$.  We denote by $\Hc_\pi^{\infty}$ the $G$-Fr\'echet module of smooth vectors
and by $\Hc_\pi^{-\infty}$ its dual. Elements in $\Hc_\pi^{-\infty}$ are called 
{\it distribution vectors}. It is known (see \cite{KS}) that for a real spherical space
the space $(\Hc_\pi ^{-\infty})^H$ of $H$-fixed distribution vectors is finite dimensional. 
The representation $\pi$ is said to be
$H$-distinguished if
$(\Hc_\pi^{-\infty})^H \ne \{0\}$.

\par Let $\eta \in (\Hc_\pi ^{-\infty})^H$ and $H_\eta<G$ the 
stabilizer of $\eta$. 
Note that $H<H_\eta$ and set $Z_\eta:= G/H_\eta$.  
With regard to $\eta$ and $v\in \Hc^\infty$ we form the generalized matrix-coefficient 
$$m_{v,\eta}(gH):= \eta(\pi(g^{-1}) v) \qquad (g\in G)$$
which is a smooth function on $Z_\eta$.

We recall the following facts from \cite{KKSS2} Thm.~7.6 and Prop.~7.7:

\begin{prop} \label{propI}Let $Z$ be a wavefront real spherical space with $H<G$ reductive.
Then the following assertions hold: 
\begin{enumerate} 
\item\label{one} Every generalized matrix coefficient $m_{v,\eta}$ as above is bounded.
\item\label{two} Let $(\pi, \Hc)$ be a unitary irreducible representation of $G$ and 
let $\eta \in (\Hc_\pi^{-\infty})^H$.  Then $Z_\eta$ is unimodular, i.e. carries a positive $G$-invariant Radon measure,    and there exists $1\leq p <\infty$ such that
$m_{v,\eta}\in L^p(Z_\eta)$ for all $v \in \Hc_\pi^\infty$. \end{enumerate}
\end{prop}

The property of $Z=G/H$ that (\ref{two}) is valid for all 
$\pi$ and $\eta $ as above is denoted {\it Property (I)} in \cite{KKSS2}. 
Note that (\ref{one}) and (\ref{two}) together imply $m_{v,\eta}\in L^{q}(Z_\eta)$ for $q>p$.
Assuming Property (I) we can then make the following notation.

\begin{definition}\label{defi p_H}
Given $\pi$ as above, we define $p_{H}(\pi)$
as the smallest index $\ge 1$ such that
all $K$-finite generalized matrix coefficients $m_{v, \eta}$
with $\eta\in (\Hc_\pi^{-\infty})^H $
belong to $L^{p}(Z_\eta)$ for any $p>p_{H}(\pi)$.
\end{definition}

Notice that $m_{v,\eta}$ belongs to $L^p(Z_\eta)$ for
all $K$-finite vectors $v$ once that this is the case for
some non-trivial such vector $v$, see \cite{KKSS2} Lemma 7.2. 
For example, this could be the trivial $K$-type,
if it exists in $\pi$.

It follows from Proposition \ref{propI} (\ref{two}) and 
finite dimensionality of
$(\Hc_\pi^{-\infty})^H$
that $p_H(\pi)<\infty$.
We say that $\pi$ is $H$-tempered if $p_{H}(\pi)=2$.
Note that if $\pi$ is not $H$-distinguished
(that is, if $(\Hc_\pi^{-\infty})^H=0$)
then $p_{H}(\pi)=1$.

\subsection{The wavefront lemma} 
Denote $A^-=\exp(\af^-)$ and notice that the wavefront property (\ref{wf}) implies that 
\begin{equation}\label{AZ}
A^-_Z\cdot z_0 = A^-\cdot z_0.
\end{equation}

With that we obtain a generalization of 
 the ``wavefront lemma'' of Eskin-McMullen 
(\cite{EM} Theorem 3.1), see Lemma 6.3 in \cite{KKSS1}. 

\begin{lemma}{\rm (Wavefront Lemma)} \label{wfl} Suppose that $Z=G/H$ is a wavefront real spherical space. 
Then there exists a closed subset $E\subset G$  with the following properties.
\begin{enumerate}
\item $E\to G/H$ is surjective.
\item The family of left translations of $Z$ by elements $g\in E$
is equicontinuous at $z_0$, that is,
for  every  neighborhood ${\mathcal V}$ of $\1$ in $G$, there exists 
a neighborhood ${\mathcal U}$ of $\1$ in $G$ such that 
$$ z\in {\mathcal U} \cdot z_0 \Rightarrow g\cdot z \in {\mathcal V}g\cdot z_0$$
for all $g\in E$.
\end{enumerate}
\end{lemma}

\begin{proof} Put  $$E=\Omega A^-\W\, .$$
Note that $E\subset G$ is closed as $A^-\subset G$ is closed and both $\Omega$ and $\W$ are compact subsets 
of $G$. Then (1) follows from the wavefront assumption, by (\ref{polar})
and (\ref{AZ}). 

Let  ${\mathcal V}$ be a neighborhood of $\1$ in $G$. By compactness 
there exists a neighborhood ${\mathcal V}_1\subset{\mathcal V}$
such that  $\Ad(x) \mathcal{V}_1\subset {\mathcal V}$ for all $x\in\Omega$.
Then it suffices to establish the implication in (2) for  $g\in A^ -\W$
and with ${\mathcal V}_1$ instead of ${\mathcal V}$.

Since conjugation by $A^-$ contracts $\nf$ there exists an open neighborhood
${\mathcal U}_1\subset \mathcal {V}_1\cap P$ of ${\bf 1}$ in $P$
such that $\Ad(a) \mathcal{U}_1\subset {\mathcal U}_1$ for all $a\in A^ -$.
It follows from the openness of $P \sw\cdot z_0$ for each $\sw$ that 
$$ \bigcap_{\sw\in\W}\, \sw^{-1} {\mathcal U}_1\sw \cdot z_0$$ 
is open. It hence contains
${\mathcal U}\cdot z_0$ for some neighborhood $\mathcal U$ of $\1$ in $G$.
With that we obtain for $z\in{\mathcal U}\cdot z_0$ 
that $\sw\cdot z\in {\mathcal U}_1\sw\cdot z_0$ for all $\sw\in\W$, and hence 
for $g=a\sw\in A^-\W$ 
that
$$g\cdot z \in 
 a  {\mathcal U}_1 \sw \cdot z_0 
\subset {\mathcal U}_1 a\sw \cdot z_0\subset {\mathcal V}_1g \cdot z_0\,.$$
as claimed.
\end{proof}

We refer to \cite{DRS} or \cite{EM} for  the notion of well-rounded balls.  

\begin{theorem}\label{MTC wf}  Let $Z=G/H$ be a wave-front real spherical space. Then MTC holds  for any family $\B$
of well-rounded balls.  If in addition $H$ is reductive, then MTC holds for the intrinsic balls. 
\end{theorem} 

\begin{proof} The first part follows from \cite{EM}. To be precise: Th. 1.2 (Equidistribution) in \cite{EM} only requires 
the wave-front lemma. MTC, that is \cite{EM} Th. 1.4, then follows from Th. 1.2 for any family of well-rounded 
balls.  The last statement follows from \cite{KSS}, Section 2,  where it was shown that the intrinsic balls are well rounded for $H$ reductive.
\end{proof} 

\begin{rmk}\label{WFgpcase}  The wavefront lemma holds for an arbitrary Lie group $G$ when considered 
as a homogeneous space  $Z=G \times G/ \Delta(G)\simeq G$ with isomorphism provided by 
the map 
$$p:Z \to G\, , \ \  (g_1, g_2) \Delta(G) \mapsto g_1g_{2}^{-1}\, .$$
We take $E=G \times \{1\} \subset G \times G.$
Now for $(g,\1) \in E$  and  ${\mathcal U}=U_1 \times U_2$, 
a neighborhood of $(\1,\1)$ in $G\times G$,   we have
$$ p((g,\1) {\mathcal U} \Delta(G)) =  g U_1 U_2^{-1}.$$  

On the other hand  for the given neighborhood  ${\mathcal V}=V_1 \times  V_2$ we have similarly
$$ p({\mathcal V} (g,\1) \Delta(G))=  V_1 g V_2^{-1}$$  and this contains $gV_2^{-1}.$
Thus we just have to require of ${\mathcal U}$ that $U_1 U_2^{-1} \subset V_2^{-1}.$
\end{rmk}

\section{Generalities on counting with error terms }\label{Sect4}

After the interlude on wavefront real spherical spaces, we pick up the discussion from Section \ref{Sect2} 
and continue 
with error terms for the main term count.  Assume that we have a quadruple $(X,\mu, D, (B_R)_{R>0})$ which satisfies 
main term counting MTC.  We then define the pointwise error term 

$$\err_{pt}(R, D):=|N_{R}(D, X) - |B_{R}||$$ 
and one might ask for the optimal $\alpha\leq 1$  such that 
$$\err_{pt}(R,D)\ll_\e |B_R|^{\alpha +\e}  \quad (\e >0)\, .$$ 

For the GCP one knows that  $\alpha\leq \frac {131}{416}=0,3149\ldots$ and some believe that $\alpha=\frac{1}{4}$, the 
lower threshold of Hardy and Landau, is possible. 

\par In this regard  we mention that $\alpha=\frac{1}{2}$ can be achieved for every family of balls which are given by 
$B_R= R \cdot B$ for some absolutely convex bounded set $B\subset \R^2$ with $C^2$-boundary, 
see \cite{Iv}, p.8.
Also  observe that the geometry of $B$ matters for the 
error count as $\alpha=\frac{1}{2}$ is obviously  optimal for the square $B=[-1,1]^2$.

\par Thus if it comes to error term bounds we should use balls which are as round as possible, 
for instance the intrinsic balls
which we just introduced. 

\par  In obtaining the reasonably good bound of $\alpha=\frac{1}{3}$ for the GCP elementary techniques 
from Fourier analysis suffice.  As an outsider one might ask at what threshold of $\alpha$  analysis converts to number theory.  
Before we come to the issue of $\alpha=\frac{1}{3}$  we pin down a more specific general setup. 

\par For the remainder we only consider 
quadruples $(G/H, \mu, \Gamma/ \Gamma_H, (B_R)_{R>0})$ which satisfy the MTC. 
Specifically  $Z=G/H$ is  a unimodular homogeneous space such that there is a lattice 
$\Gamma\subset G$ with $\Gamma_H=\Gamma\cap H$ a 
lattice in $H$. We assume that the Haar measures on $G$ and $H$ are normalized such  that
$G/\Gamma$ and $H/\Gamma_H$ both have volume 1.

\par With this set-up there is then a double fibration
$$ \xymatrix{ &G/\Gamma_H \ar[ld] \ar[rd] & \\ 
Z=G/H & & Y=G/\Gamma& }$$
By fibre-wise integration we obtain maps between functions on
$Z$ and $Y$. From $L^\infty (Y)$ to $L^\infty(Z)$ we thus have 
$\phi\mapsto \phi^H$ defined by
\begin{equation} \label{I1}
\phi^H(gH)
:=\int_{H/ \Gamma_H} \phi(gh\Gamma)\ d(h\Gamma_H)\qquad (\phi \in L^\infty(Y))\end{equation}
with
\begin{equation} \label{I1a} \|\phi^H\|_\infty \leq \|\phi\|_\infty
\, .\end{equation}
In the opposite direction, from $L^1(Z)$ to $L^1(Y)$ we have
$f\mapsto f^\Gamma$ defined by 
\begin{equation}\label{I2}  
f^\Gamma(g\Gamma):=\sum_{\gamma\in
\Gamma/ \Gamma_H} f(g\gamma H)\qquad (f \in L^1(Z))\, , \end{equation}
and likewise contractive
\begin{equation} \label{I2a} \|f^\Gamma\|_1 \leq \|f\|_1
\, .\end{equation}
{}From Fubini's theorem we obtain the following adjointness relation:
\begin{equation} \label{ufo} \la f^\Gamma, \phi \ra_{L^2(Y)}=
 \la f, \phi^H\ra_{L^2(Z)}
\end{equation}
for all $\phi\in L^\infty (Y)$ and $f\in L^1(Z)$.
In particular, applying (\ref{ufo}) to $|f|$ and $\phi=\1_{Y}$ implies
(\ref{I2a}).

We write $\1_R \in L^1(Z)$ for the characteristic function of
$B_R$ and deduce
\begin{itemize}
\item $\1_R^\Gamma(e\Gamma)=N_R(\Gamma, Z):=
\#\{ \gamma\in \Gamma/\Gamma_H\mid \gamma\cdot z_0\in B_R\}$.
\item $\|\1_R^\Gamma\|_{L^1(G/\Gamma)}= |B_R|$.
\end{itemize}
where the second equality follows from
(\ref{ufo}) with $\phi=\1_{Y}$.

\begin{ex} Once again we return to the GCP with $Z=G=\R^2$ and 
$\Gamma=\Z^2$. Then 
$$ N_R(\Z^2, \R^2) = \sum_{\gamma\in \Z^2} \1_R(\gamma)\, .$$
Let us recall how to obtain the bound 
$\alpha={1\over 3}$  by means of harmonic analysis
on $Y=\R^2/\Z^2$, more precisely
the Poisson summation formula.
Informally we would like to apply this and deduce 
\begin{equation} \label{Poisson} 
N_R(\Z^2, \R^2) = \sum_{\gamma\in\Gamma^\wedge}  \widehat{\1_R}(\gamma)\end{equation}
where $\widehat{\1_R}$ is the Fourier-transform of the the characteristic function $\1_R$ and $\Gamma^\wedge$ is 
the dual lattice. 
However, as the cutoff $\1_R$ is not smooth,  the sum in (\ref{Poisson}) is not absolutely convergent. 
The remedy is to smoothen $\1_R$ by 
convolution with some radial $\varphi\in C_c^\infty(G)$ of integral $1$, i.e.  set $\1_{R,\varphi}:= \varphi * \1_R$ and 
consider 
\begin{equation} \label{molly sum}
\sum_{\gamma\in\Z^2} \1_{R,\varphi} (\gamma) = 
\sum_{\gamma\in\Gamma^\wedge}  \widehat{\1_{R,\varphi}}(\gamma) = \widehat{ \1_{R,\varphi}} (0) + 
\sum_{\gamma\in\Gamma^\wedge\atop \gamma\neq 0}  \widehat{\1_{R,\varphi}}(\gamma) \, .
\end{equation}
Now observe that 
$$\widehat{ \1_{R,\varphi}} (0) = \int_{\R^2} \1_{R,\varphi} (x,y) \ dx\wedge dy=|B_R|$$
and  $ \sum_{\gamma\in\Z^2\atop \gamma\neq 0}  \widehat{\1_{R,\varphi}}(\gamma) $ converges. By using for $\varphi$
an approximation of the identity, one derives the estimate
as follows.

Since the theory applies equally well in $n$ dimensions we will work in this generality,
and fix a non-negative function $\varphi\in C^\infty(\R^n)$, supported in $B_1$ and with
integral $1$. 
Let $\varphi_\epsilon(x)=\epsilon^{-n}\varphi(\epsilon^{-1}x)$, and observe
that $$\1_{R-\epsilon}*\varphi_\epsilon\leq \1_R \leq \1_{R+\epsilon}*\varphi_\epsilon.$$
Hence if we denote $N_{R,\epsilon}=\sum_{\gamma\in \Gamma} \1_R*\varphi_\epsilon(\gamma)\, $
we find
\begin{equation} \label{molly ineq} 
N_{R-\epsilon,\epsilon}\leq N_R \leq N_{R+\epsilon,\epsilon}.
\end{equation}
By explicit calculation (see \cite{SW}, Ch.~IV.3) one finds 
$$\widehat{\1_R}(\lambda)=\int_{\|x\|\leq R} e^{-i\lambda\cdot x} \,dx
=C_1 R^n\int_{-1}^1 (1-t^2)^{\frac{n-1}2} e^{-i |\lambda|Rt}\, dt 
=C_2 \left(\frac R{|\lambda|}\right)^{\frac n2} J_{\frac n2}(R|\lambda|)$$
where $J_{\frac n2}$ is a Bessel function and  $C_1,C_2>0$ some constants, 
which depend only on $n$. 
It is well-known (\cite{SW}, Lemma~IV.3.11) that the Bessel functions
$J_m(x)$ for all $m\ge 0$ 
behave like $x^{-\frac12}$ as
$x\to\infty$. Hence 
$$\widehat{\1_R}(\lambda)=O(R^{\frac{n-1}2}|\lambda|^{-\frac{n+1}2})$$
as $R|\lambda|\to\infty$. 

Since $\1_{R}*\varphi_\epsilon\in C^\infty_c(\R^n)$ we can apply Poisson summation.
By using $\widehat{\1_{R}*\varphi_\epsilon}=\widehat{\1_{R}}\widehat{\varphi_\epsilon}$ we obtain
from (\ref{molly sum})
\begin{equation}\label{NR-sum}N_{R,\epsilon}=\sum _{\gamma\in\Gamma^{\wedge}}  
\widehat{\1_R}(\gamma)\widehat{\varphi}(\epsilon\gamma)=
|B_R|+\sum _{\gamma\in\Gamma^{\wedge},\gamma\neq0}  
\widehat{\1_R}(\gamma)\widehat{\varphi}(\epsilon\gamma),
\end{equation}
where $\Gamma^{\wedge}$ denotes the dual lattice.
Next we claim that there exists a constant $C>0$ such that
\begin{equation}  \label{Gamma-sum}\sum _{\gamma\in\Gamma^{\wedge},\gamma\neq 0} |\gamma|^{-\frac{n+1}2} |\widehat{\varphi}(\epsilon\gamma)|
\leq C\epsilon^{\frac{1-n}2}\end{equation}
holds for all $\epsilon>0$.  In fact,   by the rapid decay of $\widehat{\varphi}$ we have 
$|\widehat{\varphi}(\e x)|\leq C  (1 + |\e x|)^{-n -1}$  for all $x \in \R^n$ and a fixed constant 
$C>0$.  Hence, by the integral criterion for sums of decaying functions we obtain 

$$\sum _{\gamma\in\Gamma^{\wedge},\gamma\neq 0} |\gamma|^{-\frac{n+1}2} |\widehat{\varphi}(\epsilon\gamma)|
\leq C \underbrace{\int_{|x|>c}   |x|^{-\frac{n+1}2} (1+|\e x|)^{-n -1} \ dx}_{=: I_\e} $$
 for a constant $c=c(\Gamma^{\wedge})>0$. 
 Finally 
 
$$ I_\e = \e^{\frac {1-n}2}\int_{|x|> c\e}  |x|^{-\frac{n+1}2} (1+|x|)^{-n -1} \  dx $$
and the claim follows from $n\geq 2$.
Combining (\ref{NR-sum})  with (\ref{Gamma-sum})  yields

\begin{equation} \label{eucl err1} \big|N_{R,\epsilon}-|B_R|\big|\leq C R^{\frac{n-1}2}\epsilon^{\frac{1-n}2}. \end{equation}
Now $\big||B_{R\pm\epsilon}|-|B_R|\big| \sim R^{n-1}\epsilon$, and
by comparing with (\ref{molly ineq}) we then conclude that
\begin{equation} \label{eucl err2} \big|N_{R}-|B_R|\big| \leq C\big( (R+\epsilon)^{\frac{n-1}2}\epsilon^{\frac{1-n}2}+R^{n-1}\epsilon \big)\end{equation} 
for some $C>0$ and all $\epsilon>0$. 
The best possible value of $\epsilon$ is $R^{\frac{1-n}{n+1}}$ and yields
\begin{equation}\label{eucl err3} \big|N_{R}-|B_R|\big|\leq C R^{\frac{n(n-1)}{n+1}}. \end{equation} 

For $n=2$ this gives the mentioned bound by $O(R^{\frac23})$. For $n=3$ it gives
$O(R^{\frac32})$, compared to the best known upper bound $O(R^{\frac43}\log R)$.
For $n\ge 4$ the error is known to behave essentially like $R^{n-2}$
(see \cite{Fricker}).
\end{ex}

\par This method resembles to some extend the approach in \cite{KSS} where error term bounds were obtained 
for wavefront spaces: Poisson summation was replaced by spectral analysis and Weyl's law.  

\begin{ex}
As a second instance of a classical problem of counting lattice points we review the case of the hyperbolic 
space. This study was initiated by Delsarte, see \cite{Del}.
Let the upper half plane
$\mathbf{H}=\{z=x+iy \in \mathbb{C}: y>0\}$ be  equipped with the invariant metric
$ds^2=\frac{dx^2+dy^2}{y^2}$ by which it has  constant negative curvature. 
The Riemanian metric induces a distance function $d$ on $\mathbf{H}$ which we use to define a family 
of balls $$B(R)=\{z \in \bH: d(z,i)<R\}\qquad  (R>0).$$
Their volume is exponentially growing with the radius:   
$$\vol(B(R)):=\int_{B(R)} \frac{dx  dy}{y^2}=4 \pi \sinh^2(R/2)\sim \pi e^R\, .$$

To set up the counting problem, recall that the hyperbolic plane admits a group of symmetries 
$G:=\Iso(\mathbf{H})=\PSL_{2}(\mathbb{R})$ and the role of the lattice is played by a discrete subgroup 
$\Gamma < \PSL_{2}(\mathbb{R})$ of finite co-volume.
More precisely,  we consider the set $D_{\Gamma}(z)=\Gamma \cdot z,$ with $z \in \mathbf{H}.$
The hyperbolic lattice counting problem then consists of estimating 
$$N_{R}(z):=\# D_{\Gamma}(z) \cap B(R)=\#\{\gamma \in \Gamma: d(i,\gamma \cdot z)<R\}.$$
as $R\to\infty$.

An interesting feature of this problem is that, unlike the Euclidean lattice point counting, a direct 
packing argument is not possible since most of the volume of a hyperbolic ball is located near its boundary. 
This so-called mass concentration phenomenon is a reflection of the exponential volume growth of the Haar 
measure on the semi-simple Lie group $G$. 
Thus, different techniques are required to 
approach this lattice counting problem and even obtaining the main term
was a non-trivial achievement.
In \cite{Del} the problem is studied for co-compact lattices and 
MTC is proved, that is, with the measure normalized as above, 
$$N_{R}(z) \sim \frac{1}{\vol(\Gamma \backslash \mathbf{H})}\,\pi e^R$$
for all $z$, as $R\to\infty$.
In \cite{Sel} Selberg developed the trace formula as a tool of
obtaining geometric information from spectral information and vice-versa. 
\par
In case $\Gamma=\PSL(2,\Z)$, Selberg showed  \cite{Sel} 
that 
$$N_{R}(i) \sim \frac{3}{\pi} \vol(B(R))$$ 
and obtained an error term estimate, 
$$|N_{R}(i)- \frac{3}{\pi} \vol(B(R))|=O(\exp(2R/3)).$$ 
Maybe the most significant result is the connection between the 
lattice counting problem for a general lattice of finite co-volume $\Gamma \subset \PSl(2,\Z)$ 
and the so-called automorphic spectrum. More precsiely, Selberg provided the following
asympotic formula expressing the number of lattice points as a sum over
the eigenvalues of the Laplace operator on the Riemann surface $\Gamma \backslash \mathbf{H}=\Gamma \backslash G/K.$ 
To formulate the exact formula we denote by $\Delta$ the Laplace-Beltrami operator on the Riemannian manifold 
$\Gamma \backslash \mathbf{H}.$ This is sometimes called the hyperbolic Laplacian. In case 
$\Gamma \backslash \mathbf{H}$ is compact the spectrum of $\Delta$ is discrete. In the non-compact case, 
in addition to the discrete spectrum there is a contribution from the continuous spectum. This 
spectrum consists of Eisenstein series which are parametrized by the interval 
$[\frac{1}{2},\infty)$ and by cusps. 
In both cases, we denote by $\{u_{j}\}_{j \geq 0}$ a complete system of orthonormal eigenfunctions for the 
discrete spectrum of the hyperbolic Laplacian corresponding to 
$0=\lambda_{0} < \lambda_{1} \leq \lambda_{2} \leq \cdots .$ 
In particular $u_{0}=\vol(\Gamma \backslash \mathbf{H})^{-\frac12}$.

Fix $z,w \in \mathbf{H}$ and let 
$$N(R,z,w)=\#\{ \gamma \in \Gamma: 2 \cosh(d(\gamma z, w))<R\},$$ 
thus $N_{R}(z)=N(R',z,i)$ with $R'=2 \cosh(R) \sim e^{R}$. By using the spectral expansion of the automorphic kernel (see Theorem 7.1 \cite{Iwa}), Selberg showed in \cite{Selb} and 
\cite{SelCol} (see also Theorem 2 of \cite{Pa} and Theorem 12.1 in \cite{Iwa}) the following formula:
$$N(R,z,w)= \sqrt{\pi} \sum_{\frac{1}{2}  <  s_{j} \leq 1} \frac{\Gamma(s_j-\frac{1}{2})}{\Gamma(s_j+1)} u_{j}(z)\overline{u_{j}(w)}R^{s_j} + O(R^{\frac{2}{3}})$$
The relation between $s_{j}$ and the eigenvalue $\lambda_{j}$ is given by the usual translation $\lambda_{j}:=s_{j}(1-s_{j}).$
Notice that in the sum, the value $s_{0}=1$ corresponds to the constant function yielding the main term 
$\frac{\pi}{\vol(\Gamma \backslash \mathbf{H})}R.$ The rest of the sum is over exceptional eigenvalues, 
that is those eigenvalues 
of the Laplacian that satisfy $\lambda_{j}:=s_{j}(1-s_{j})<\frac{1}{4}.$

The moral of the last example is that even if one invokes the full force of the spectral decomposition 
of $L^{2}(\Gamma \backslash \mathbf{H})$ and the trace formula, understanding the error term in the 
lattice counting problem requires refined information about the automorphic spectrum. For example, 
when considering congruence lattices $\Gamma(N)$, the Selberg-$\frac{1}{4}$-conjecture, that there are 
no exceptional eigenvalues, that is $\lambda_{1}(\Gamma(N)\backslash \mathbf{H}) \geq \frac{1}{4}$, 
is relevant to determining the main term in the corresponding lattice point counting formula.
\end{ex}
\bigskip 
\par For the rest of this article we assume now that $Z=G/H$ is a wavefront real spherical space, $H$ is reductive 
and $\B$ is constituted by the intrinsic balls. Further we restrict ourselves to the cases where the cycle $H/\Gamma_H$
is compact.\footnote{After a theory for regularization of $H$-periods of Eisenstein series is developed, 
 it is expected that one can drop this assumption.} 
To simplify the exposition here we assume in addition that $\Gamma<G$ is
irreducible, i.e. there do not exist non-trivial  normal subgroups $G_1, G_2$  of $G$ and 
lattices $\Gamma_i < G_i$ such that $\Gamma_1 \Gamma_2$ has finite index in $\Gamma$. 

\par The error we study is measure theoretic in nature,
and will be denoted here as $\err(R,\Gamma)$.
Thus, $\err(R,\Gamma)$ measures the deviation of two measures on
$Y=\Gamma \backslash G$,
the counting measure arising from lattice points in a ball of radius $R$,
and the invariant measure $d \mu_{Y}$ on $Y$. More precisely, 
with $\1_R$ denoting the characteristic function of $B_R$ we
consider the densities $$F^{\Gamma}_{R}(g \Gamma):=
\frac{\sum_{\gamma\in \Gamma/ \Gamma_H} \1_R(g\gamma H)}{|B_{R}|}.$$   
Then, 
$$\err(R,\Gamma)=||F^{\Gamma}_{R}-d \mu_{Y}||_{1},$$ 
where  $||\,\cdot\,||_{1}$ denotes the total variation of the signed measure. Notice that 
$|F^{\Gamma}_{R}(e\Gamma)-1|=\frac{|N_{R}(\Gamma, Z) - |B_{R}||}{|B_{R}|}$ is essentially 
the error term for the point wise count. 

Our results on the error term $\err(R,\Gamma)$ allows us to deduce results toward the error 
term in the smooth counting problem, a classical problem that studies the quantity 
 $$\err_{pt,\alpha}(R,\Gamma)=|B_R||F^{\Gamma}_{\alpha,R}(e\Gamma)-1|$$ where 
 $\alpha \in C_{c}^{\infty}(G)$ is a positive smooth function of compact support (with integral one) 
 and $F^{\Gamma}_{\alpha,R}=\alpha*F^{\Gamma}_{R}.$
We refer to \cite{KSS}, Remark 5.2 for the comparison of $\err(R,\Gamma)$ with 
$\err_{pt,\alpha}(R,\Gamma)$.

\par Below we will introduce an exponent $p_H(\Gamma)$ (see (\ref{ic})), which measures the worst $L^{p}$-behavior
of any generalized matrix coefficient associated with a spherical
unitary representation $\pi$, which is  $H$-distinguished and occurs
in the automorphic spectrum of $L^2(\Gamma \backslash G)$. 
In \cite{KSS} we obtained
the following error bound for triple spherical spaces.

\begin{theorem}\label{thmB} 
Let $Z=G_0^3/ \diag (G_0)$ for $G_0=\SO_e(1,n)$ 
and assume that $H/\Gamma_H$ is compact.
For all $p>p_H(\Gamma)$ there exists a $C=C(p)>0$ such that
$$\err(R,\Gamma)\leq C |B_R|^{-{1\over (6n+3)p}}\, $$
for all $R\geq 1$.
\end{theorem}

\section{Counting with error terms on a wavefront real spherical space}\label{Sect5}  
We assume from now on that $G$ is semisimple with no compact factors.

First we need some notation.  We denote by $\hat G$ the unitary dual of $G$, i.e. the set of equivalence classes 
of irreducible unitary representations of $G$.   We recall the maximal compact subgroup $K<G$. By 
$\hat G_s\subset \hat G$ we understand the subset which corresponds to $K$-spherical representations, i.e. 
representations which have a non-zero $K$-fixed vector.  

With that 
we define in continuation of Definition \ref{defi p_H}
\begin{equation} \label{ic} p_H(\Gamma):= \sup\{ p_H(\pi) \mid \pi \in \hat G_s\cap  
\supp L^2(G/\Gamma)\} \end{equation}
and recall that $p_H(\Gamma)<\infty$ by \cite{KSS}, Lemma 6.3. 

A subset $\Lambda\subset \hat G_s$  is called {\it $L^p$-bounded}  provided that 
$p_H(\pi) < p$ for all $\pi \in \Lambda$.   Set $\Lambda_p:=\{ \pi \in \hat G_s\mid   p_H(\pi)<p\}$.
We note that 
$$\Lambda_p=\{ \pi\in\hat G_s\mid(\forall \eta\in (\Hc_\pi^{-\infty})^H ) (\forall v\in \Hc_\pi^{K-{\rm fin}})  
\ m_{v,\eta}\in L^p(Z_\eta)\}$$
and 
$$\Lambda_p \subset \Lambda_{p'} \qquad  (p<p')\, .$$

In \cite{KSS} we formulated the following hypothesis.

\bigskip 
\par\noindent{\bf Hypothesis A:} {\it  For $1\leq p<\infty$
there exists a compact subset $D\subset G$ and constants $c, C>0$ 
such that the following assertions hold
for all $\pi\in \Lambda_p$, $\eta\in (\Hc_\pi^{-\infty})^H$ and
$v\in \Hc_\pi^K$:
\begin{equation}\tag{A1} \label{hypo1}  \|m_{v, \eta}\|_{L^p(Z_\eta)}  \leq C  
\|m_{v, \eta}\|_\infty\, ,\end{equation}
\begin{equation}\tag{A2} \label{hypo2} \|m_{v, \eta}\|_\infty  \leq c \| m_{v, \eta}\|_{\infty, D}
\,  . \end{equation}
Here $ \| \,\cdot\, \|_{\infty, D}$ denotes the supremum norm taken on the subset $D.$ }

\medskip
Under this hypothesis we have then shown in \cite{KSS} Thm.~7.3 that:  

\begin{theorem}\label{thmC} 
Let $Z$ be wavefront real spherical space 
for which Hypothesis A  is valid.
Assume also
\begin{itemize}
\item $G$ is  semisimple with no compact factors
\item $H$ is reductive and the balls $B_R$ are intrinsic
\item $\Gamma$ is an arithmetic and irreducible lattice
\item $\Gamma_{H}=H \cap \Gamma$ is co-compact in $H$
\item $p>p_H(\Gamma)$
\item $k>\frac12(\operatorname{rank}(G/K)+1)\dim(G/K) +1$
\end{itemize}
Then there exists a constant $C=C(p, k)>0$ such that
\begin{equation}\label{error bound}
\err(R,\Gamma)\leq C |B_R|^{-{1\over (2k+1)p}}\, 
\end{equation}
for all $R\geq 1$.
Moreover, if $Y=\Gamma \backslash G$ is compact one can replace the last  condition by $k>\dim(G/K)+1$.
\end{theorem}

In \cite{KSS} the Hypothesis A was verified for the triple spaces,
and thus we could derive Theorem \ref{thmB} from Theorem \ref{thmC}.

When writing \cite{KSS} the harmonic analysis on real spherical spaces was not sufficiently developed to 
obtain Hypothesis A  (or suitable variants thereof) in general.   In particular 
at that time we were not able to derive an error bound 
for all wavefront real spherical spaces. For symmetric spaces
the existence of a non-quantitative error
term was established in
\cite{BeOh} and improved in \cite{GN}. 
\par In case of the hyperbolic plane our error term (\ref{error bound}) 
is still
far from the quality of the bound of A. Selberg. This is because the approach in \cite{KSS} 
only uses a weak version of the trace formula, namely Weyl's law, and
uses simple soft Sobolev bounds between eigenfunctions on~$Y$.

Here we shall prove a slightly weaker version of Hypothesis A
for spaces $Z$ of real rank one. In order to prepare for it we draw from some notions
which we used in the lecture notes \cite{KS2}.  Together with every 
G-continuous norm $p$ on a Harish-Chandra module
$V$ there comes a family of Laplace-Sobolev norms $(p_k)_{k\in \N}$ on $V^\infty$. 
We briefly recall the definition of the $p_k$ for $k$ even (one can define Sobolev norms $p_k$ for any value 
of $k\in\R_{\geq 0}$). 
For that we let $\Delta\in\U(\gf)$ be a Laplace element such that $\Delta=\Cc+2\Delta_\kf$ where 
$\Cc$ is the Casimir element of $\gf$ and $\Delta_\kf\in \U(\kf)$ is a Laplace element.  Then 
\begin{equation}\label{p2k} p_{2k}(v):=\sum_{j=0}^k p (\Delta^j v) \qquad (v\in V^\infty)\, .\end{equation} 
\par For an irreducible Harish-Chandra module $V$ we let  $\chi_V\in \C$ be the multiple by  which the Casimir 
element $\Cc\in \Zc(\gf)$ acts on $V$. In particular, if $V$ is the Harish-Chandra module
 of a unitary irreducible representation $(\pi, \Hc)$ we write $|\pi| = |\chi_V|$. 
 The following lemma will be used frequently
in the sequel to relate between $p_k$ and $p$.

\begin{lemma}  \label{K-squeeze}Let $V$ be an irreducible Harish-Chandra module. Let $k\in \N_0$. 
Then there exists constants $C_1, C_2>0$, independent of $V$, such that 
\begin{equation}\label{Sobinf} C_1 (1+|\chi_V|)^{k\over 2} p(v) \leq 
p_k(v)\leq C _2 (1+|\chi_V|)^{k\over 2}  p(v)\qquad (v\in V^K)\, .
\end{equation} 
\end{lemma} 

\begin{proof}  (For $k$ even) As $v\in V^K$ is $K$-fixed we have $\Delta_\kf v=0$ and thus 
for every $j\in\N$
$$\Delta^j v = (\Cc +\Delta_\kf)^jv  = \chi_V^j v\, .$$    
Hence the assertion follows from (\ref{p2k}). 
\end{proof}

In the weaker version of Hypothesis A 
we replace (\ref{hypo1}) by a Sobolev estimate,
namely

\bigskip 
\par\noindent{\bf Hypothesis B:} {\it  Let  $1\leq p'<p<\infty$.  
Then there exist $C>0$ and $l=l(p,p')\in\N$ 
such that 
\begin{equation}\tag{B1}\label{hypoB1}  
\|m_{v, \eta}\|_{L^{p}(Z_\eta)}  \leq C  
\|m_{v, \eta}\|_{\infty,l}\, .\end{equation}
for all $\pi\in \Lambda_{p'}$, $\eta\in (\Hc_\pi^{-\infty})^H$ and
$v\in \Hc_\pi^K$.}

\begin{rmk}\label{compact quotient}
Assume $\pi$ is non-trivial. 
Under the assumptions in Thm.~\ref{thmC} 
it follows from \cite{KSS} Lemma 6.2 that 
 $H_\eta/H$ is compact, and hence $Z_\eta$ can be replaced 
 by $Z$ both in (\ref{hypo1})
 and in (\ref{hypoB1}).
 \end{rmk}

 \begin{rmk}\label{about p in Hypo B}
Let $1\le p'<p_1< p_2<\infty$.
If (\ref{hypoB1}) holds for the pair $(p_1,p')$, 
for some constant $C=C_1>0$, then it
holds for the pair $(p_2,p')$ as well, for some 
$C=C_2>0$. This follows from the fact that
$$\|f\|_{p_2}^{p_2}= 
\|f\|_\infty^{p_2} \left \|\frac f{\|f\|_\infty}\right\|_{p_2}^{p_2}
\leq
\|f\|_\infty^{p_2} \left \|\frac f{\|f\|_\infty}\right\|_{p_1}^{p_1}
=
\|f\|_\infty^{p_2-p_1}\|f\|_{p_1}^{p_1}$$
for all bounded functions $f$, and thus
$$ \|m_{v,\eta}\|_{p_2}^{p_2}
\leq \|m_{v,\eta}\|_{\infty}^{p_2-p_1}\|m_{v,\eta}\|_{p_1}^{p_1}
\leq \|m_{v,\eta}\|_{\infty,l}^{p_2-p_1} \left( C\|m_{v,\eta}\|_{\infty,l}\right)^{p_1}
=C^{p_1}\|m_{v,\eta}\|_{\infty,l}^{p_2}.$$
 \end{rmk}

 In the hypothesis above it is 
unnecessary to include an analogue of 
(\ref{hypo2}), as such an analogue can be derived
from (\ref{hypoB1}). This is the content of the next lemma.
For $R>0$ we set 
 $$A^-_R:=\{ a\in A^-\mid  \rho_\nf (\log a)  \geq -R\}\, $$ 
 where $ \rho_\nf=\frac12\operatorname{tr}\ad_\nf\in\af^*$.
 Further we set 
 $$A_{Z,R}^- :=  A_R^- A_H / A_H \subset A_Z^-$$
 and if $D\subset G$ is a compact set
 $$D_R:= D A_{Z,R}^- \W \cdot z_0 = D A_R^-  \W \cdot  z_0\subset Z\, .$$

\begin{lemma} \label{d formula}
Let $Z$ be a wavefront real spherical space for which Hypothesis B is valid.
With $p'$, $p$, and $l=l(p,p')$ as above let
\begin{equation}  \label{d-form} d=\frac{1}{4} (lp + \dim\af_Z (l+\dim \gf +1))\, .\end{equation}
 Then there exists  a compact subset $D\subset G$ 
 such that
\begin{equation}\tag{B2} \label{hypoB2} 
\|m_{v, \eta}\|_\infty = \| m_{v, \eta}\|_{\infty,  D_{d \log (1 +|\pi|)}}
\,   \end{equation} for all $\pi\in \Lambda_{p'}$, $\eta\in (\Hc_\pi^{-\infty})^H$ and
$v\in \Hc_\pi^K$.
\end{lemma}

\begin{proof}\footnote{ We allow the same symbol $C$ for
universal positive constants, independently of their actual values.}  
We recall the polar decomposition $Z=\Omega A_Z^- \W\cdot z_0$
in (\ref{polar}) and the volume bound (\ref{vbound}) for $\vb$.  Since $Z$ is wavefront 
we note that $\vb$ is increasing to infinity.
\par For $f=m_{v,\eta}$ and $g\in \Omega$ we set $f_g(z)=f(g\cdot z)$ for $z\in Z$.
We normalize $\|f\|_\infty=1$. Let now 
$g_0\in \Omega$,  $\sw\in \W$ and $X_0\in \af_Z^-$ be such that 
$|f_{g_0}(\exp(X_0) \sw\cdot z_0)|=\|f\|_{\infty}$ (these elements exist as $f$ decays at infinity by 
(\ref{vbound}) and (\ref{phi sob}) below).
We recall the invariant Sobolev Lemma from (\ref{ISL}): 

\begin{equation}\label{phi sob} |\phi(z)| \leq C   \vb(z)^{-{1\over p}} \| \phi\|_{p, s} \qquad (z\in Z) \end{equation}
for $s>{\dim G\over p}$ and all $\phi\in L^p(Z)^\infty$.

Thus (\ref{phi sob}) applied specifically to $\phi=L_Y f_g$ for $g\in \Omega$,  
$Y\in \gf$ with $\|Y\|=1$ and $z=\exp(X)\sw\cdot z_0$ 
with $X\in\af_Z^-$ yields

\begin{equation} \label{LYder}  |L_Y f_g(\exp(X)\sw\cdot z_0)|
\leq  C e^{{2\over p} \rho(X)} \|f_g \|_{p, s+1}\leq 
C  \|f\|_{p, s+1} \, \end{equation} 
for some $C>0$.  For the second inequality we used $\rho(X)\leq 0$  (as $Z$ as wavefront) and 
$\sup_{g\in \Omega} \| L_g f \|_{p, s+1}\leq C \|f\|_{p, s+1}$ for all $f\in L^p(Z)^\infty$ and a constant $C$ which only 
depends on $\Omega$ and $s$. The latter follows as the Laplace-Sobolev norms from (\ref{p2k}) compare to the standard 
Sobolev norms (see \cite{BK}, Prop. 3.5).

Now define a function on $K\times A_Z^{-}$  by

$$ F(k, X)= f_{g_0}(k\exp(X) \sw\cdot z_0)$$
and observe that (\ref{LYder}) implies 

$$ \|dF (k, X)\|\leq  C \|f\|_{p, s+1} \qquad (k\in K, X\in \af_Z^-)\, .$$
With (\ref{hypoB1}) and (\ref{Sobinf})
we thus obtain 
$$ \|d F(k, X) \| \leq  C \|f\|_{\infty, l+s+1} \leq C' 
(1 +|\pi|)^{l+s+1\over 2}\, $$
for some positive constant $C'$.
Set $\delta:={1\over 4 C'} (1 +|\pi|)^{-{l+s+1\over 2}}$. Then the mean value theorem 
implies that  there exists a
neighborhood $U$ of $\1$ in $K$ such that 

 \begin{equation} \label{F-K-U} |F(k, X)|\geq \frac12\qquad (k\in U, |X-X_0|<\delta)\, . \end{equation} 

\par From Lemma \ref{lemma int-ineq} (\ref{intineq2}) we obtain

\begin{equation} \label{int1}\int_{K}\int_{A_Z^-}  |\phi(ka\sw\cdot z_0)|^p   a^{-2\rho}   \ da \ dk \leq C \|\phi\|_p^p \qquad (\phi\in L^p(Z))\, .\end{equation}

\par  We obtain from  (\ref{F-K-U}) 
and by applying (\ref{int1}) to $\phi=f_{g_0}$ a constant $C>0$ such that 
$$ \|f\|_p\geq C \left[\vol_K(U) \vol_{\af_Z} (\{ |X-X_0|\leq \delta\} )  e^{-2\rho(X_0)}\right]^{1\over p}\, ,$$
or equivalently (by our choice of $\delta$)
\begin{equation} \label{f-p-C} \|f\|_p \geq C (1+|\pi|)^{ - \dim \af_Z {{l+s+1}\over 2p}} e^{-\frac 2{p}\rho(X_0)}\, .\end{equation} 
Now (\ref{hypoB1}) together with Lemma \ref{K-squeeze} (with the $G$-continuous norm $p$ given by 
$p(v):=\sup_{z\in Z} |m_{v,\eta}(z)|$ for $v\in V$) yields 
$$\|f\|_p\leq C\| f\|_{\infty, l} \leq C ( 1+|\pi|)^{l\over 2} \|f\|_\infty= C  ( 1+|\pi|)^{l\over 2} \, .$$
Inserting this inequality into (\ref{f-p-C}) gives 
$$e^{|\frac 2{p}\rho(X_0)|}=\e^{-\frac 2{p}\rho(X_0)}\leq C(1+|\pi|)^{{l\over 2}+ \dim \af_Z {{l+s+1}\over 2p}}$$
which shows
$$|\rho(X_0)|\leq d \log(1+|\pi|)  +C $$ 
for a constant $C$ independent of $\pi$ and 
$$d=\frac{1}{4} (lp + \dim\af_Z (l+s+1))\, .$$
This shows that the maximum of $|f|$ is attained in a region as asserted (with 
$D=\Omega A^-_C$).
\end{proof}

\begin{theorem}\label{thm error bound}
Let $Z$ be a wavefront real spherical space for which Hypothesis B is valid, 
and assume all bulleted items in Thm.~\ref{thmC}.
Let $l=l(p,p')$ for $p':=p_H(\Gamma)$ and $p>p'$ and let  $d$ be as in (\ref{d-form}). 
Then there exists a constant $C=C(p)>0$ such that 
\begin{equation}\label{new error bound}
 \err(R,\Gamma)\leq C |B_R|^{-{1\over (2s'+1)p}}\, 
\end{equation}
for all $R\geq 1$, where $s':= \dim G/K +l + 2d$.
\end{theorem}
The bound (\ref{error bound}) is obtained in \cite{KSS} Thm.~7.3
from the estimate of 
Prop.~6.5   via standard techniques quite in resemblance to the Euclidean case, see (\ref{eucl err1}) - 
(\ref{eucl err3}).  With
Hypothesis A replaced by B, the analogue of Prop.~6.5 becomes:

\begin{prop}\label{mpro} Assume that $Z$ is wavefront real spherical and all the bulleted 
assumptions from Thm.~\ref{thmC}. Assume moreover, that Hypothesis B is valid. Let

$$C^{\infty}_b(Y)^K_{\rm van}:=\{ f \in C^\infty(Y)\mid f \ \hbox{bounded, $K$-invariant and $\int_Y f(y) \ dy  =0$}\}\, .$$

Let $p>p_H(\Gamma)$. Then the  map
$${\rm Av_{H}}:
C^{\infty}_b(Y)^K_{\rm van} \to L^{p}(Z)^K, \ \ \phi\mapsto \phi^{H}; \ \phi^H(gH) =\int_{H/H_\Gamma} \phi(gh) \ d(hH_\Gamma)$$
is continuous. More precisely, let $s:=\dim( G/K)$ and $r:=\operatorname{rank} (G/K)$. Then
for all
\begin{enumerate}
\item $k> s+1$ if $Y$ is compact.
\item $k> \frac{r+1}{2} s  +1$ if $Y$ is non-compact and $\Gamma$ is arithmetic
\end{enumerate}
there exists a constant
$C=C(p, k)>0$ such that
$$\|\phi^{H}\|_{L^p(Z)} \leq
C \|\phi\|_{\infty,k+l +2d} \qquad (\phi \in
C^{\infty}_b(Y)_{\rm van}^K)\, .$$
Here $l=l(p, p_H(\Gamma))$ is the constant in Hypothesis B and the constant $d$ is from Lemma \ref{d formula}.
\end{prop}

\begin{proof} For the sake of completeness we give the slightly modified proof. In addition we take the opportunity 
to correct a few mistakes in the proof of Prop. 6.5 of  \cite {KSS}.

To begin with we recall part of the  Langlands Plancherel-Theorem (see \cite{L} or \cite{A}, p. 256, Main Theorem 
(adelic setup)) for $L^2(Y)$ which is of relevance to us, i.e. 
the $G$-invariant subspace $\Hc\subset L^2(Y)$ which is generated by the $K$-spherical spectrum. 
For the formulation of the Plancherel Theorem we restrict ourselves to the arithmetic setup, i.e. $G$ is 
defined over $\Q$ and $\Gamma$ is commensurable to $G(\Z)$.  In the sequel we use the language from \cite{Bo} and let $P_1, \ldots, P_\ell $ be 
a complete set of representatives of the $\Gamma$-conjugacy classes of minimal $\Q$-parabolic 
subgroups (cusps). Further for each $1\leq j\leq \ell $ we let $[P_j]$ be the finite set of parabolic subgroups 
which contain $P_j$. One calls $P,Q\in[P_j]$ associated provided that their Levi-subgroups are conjugate.
We write $[P_j]_{\rm Ass}$ for the corresponding association classes and set ${\mathcal P}=\bigcup_{j=1}^\ell [P_j]_
{\rm Ass}$.  Tacitly we pick in the sequel for each association class a representative and by slight abuse of notation write 
$P\in {\mathcal P}$ understanding  that $P$ is parabolic subgroup of $G$.   

\par The rough Plancherel decomposition is $\Hc= \bigoplus_{P\in{\mathcal P}}  \Hc_P$
with $\Hc_d:=\Hc_G$ yielding the discrete spectrum. 
We let $P=M_P A_P N_P$ be the associated  Langlands decomposition and recall that the projection 
of $\Gamma\cap P$ to $M_P\simeq P/ A_P N_P$ yields a lattice in $M_P$ (which is anisotropic if $P=P_j$ is minimal). 
Further we denote by $\hat M_{P, d}\subset \hat M$ the $K\cap M_P$-spherical discrete  spectrum.  
For each $\sigma\in \hat M_{P,d}$ we denote by $m(\sigma)\in\N$ the multiplicity of 
$\sigma$ in $L^2(M/ \Gamma_{M_P})$.   For $\sigma\in \hat M_{P, d}$ and $\lambda\in i\af_P^*$
we denote by $\pi_{\sigma,\lambda} =\operatorname{Ind}_P^G (\sigma\otimes \lambda)$ the associated 
unitary principal series (normalized induction).  Then one has a unitary equivalence 
of $G$-modules:

\begin{equation}\label{L-PT}\Hc_P \simeq  \bigoplus_{\sigma\in \hat M_{P, d}} m(\sigma)  \int_{i(\af_P^*)^+}^\oplus \pi_{\sigma,\lambda} \ d\lambda\, .\end{equation}

Let now  $\phi\in C_b^\infty(Y)^K_{\rm van}$ and write
$\phi=\phi_d +\phi_c$ for its decomposition in discrete and continuous
Plancherel parts. 
Observe that in general we do not have $\phi_d, \phi_c\in C_b^\infty(Y)^K_{\rm van}$ as $\phi_d$ might not be bounded in presence of non-trivial residual spectrum. However, as the arguments below will show, this will not cause 
us technical problems. 
 We consider first the discrete part $\phi_d$.

\par If $Y$ is compact it follows from Weyl's law (see \cite{Sh}, Th. 15.2) that
the multiplicities in $L^2(Y)$ satisfy 
$$ \sum_{|\pi|\leq R} m(\pi)
 \sim c_Y R^{s/2} \qquad (R\to \infty)\, ,$$
for some constant $c_Y>0$. 
It follows that
\begin{equation}\label{weight}
\sum_{\pi} m(\pi) (1+|\pi|)^{-m} < \infty
\end{equation}
for all $m> s/2+ 1$.
In case $Y$ is non-compact, we let $\hat G_{\Gamma,d}\subset \hat G$ be the
the discrete support of the Plancherel measure of $L^2(Y)$
and $m(\pi)$ the corresponding multiplicity of $\pi$. 
Assuming $\Gamma$ is arithmetic it is shown in \cite{Ji} that
\begin{equation}\label{mult1} \sum_{\pi \in \hat G_{\Gamma,d} \atop |\pi|\leq R} m(\pi)
 \leq  c_Y R^{rs/2} \qquad (R>0)\, \end{equation}
so that for $m> rs/2+ 1$ we again obtain (\ref{weight}).

\par Let $p>p_H(\Gamma)$.
As $\phi_d$ is in the discrete spectrum we  decompose it
as $\phi_d=\sum_\pi \phi_{\pi}$ according to $\Hc_d \simeq\bigoplus_{\pi} m(\pi) \Hc_\pi$, i.e. 
$\phi_\pi$ is the $\pi$-isotypical part of $\phi_d$.   Since $\dim \Hc_\pi^K \leq1$ and $\phi_\pi$ 
is $K$-fixed, we observe that 
$\phi_\pi\in m(\pi)\Hc_\pi$ generates an irreducible $G$-submodule of $m(\pi)\Hc_\pi$ 
which is isomorphic to $\Hc_\pi$ or $\{0\}$ according to whether $\phi_\pi$ is zero or not. In particular, 
we obtain with (\ref{hypoB1}) that 
$$\|\phi_d^{H}\|_{p} \leq \sum_\pi \|\phi_{\pi}^{H}\|_{p}
\leq C \sum_\pi \|\phi_{\pi}^H\|_{\infty, l}\, .$$
Note that $\pi$ is non-trivial in these sums since $\phi_d$ has vanishing integral (this is the only place
where the vanishing hypothesis on $\phi$ enters the proof). 
With (\ref{hypoB2}) we obtain further
$$\|\phi_d^{H}\|_{p} \leq C \sum_\pi \|\phi_{\pi}^H\|_{\infty,  D_{d\log(1+|\pi|)}, l}\, .$$
Let $B_{|\pi|} = D A_{ d \log (1 +|\pi|)} ^-  \W  H_c\subset G$
where $H_c\subset H$ is a compact subset such that $H_c\Gamma_H =H$. 
Note that $\|\phi_\pi^H\|_{\infty, \Omega_{d\log(1+|\pi|)}}\leq \|\phi_\pi\|_{\infty, B_{|\pi|}}$ which 
together with (\ref{Sobinf})
allows us to estimate the last sum as follows:

$$
\sum_\pi \|\phi_{\pi}^H\|_{\infty, D_{d\log(1+|\pi|)},l}\leq \sum_\pi
\|\phi_{\pi}\|_{\infty, B_{|\pi|}, l}
\leq C \sum_\pi  (1+|\pi|)^{-m/2}\|\phi_{\pi}\|_{\infty,B_{|\pi|}, m+l}
$$
with $C>0$ a constant depending only on $l$, where we used Lemma \ref{K-squeeze} for the last inequality.

Applying the Cauchy-Schwartz inequality combined with (\ref{weight})
we obtain
\begin{equation} \label{b1}\|\phi_d^{H}\|_{p} \leq C \Big(\sum_{\pi} \|\phi_{\pi}\|_{\infty,B_{|\pi|}, m+l}^{2}\Big)^{\frac{1}{2}}\, 
\end{equation}
 with $C>0$. 

\par  In the sequel we view functions on $Y$ as right $\Gamma$-invariant functions on $G$. 
Recall from \cite{KS2}, Sect. 4,  the notion of volume weights on a homogeneous space for $G$ and let 
$\vb_Y$ be a volume weight for $Y=G/\Gamma$. 
{}From the definition of $B_{|\pi|}$ and Lemma \ref{Siegel lemma} below  applied to $\Omega=\W H_c$
we now infer 
\begin{equation} \label{S-bound} \inf \vb_Y|_{B_{|\pi|}}  \geq  C  ( 1+|\pi|)^{-2d}  \qquad  (\pi \in \hat G_s)\, .
\end{equation}
The invariant Sobolev lemma (see \cite{KS2}, Lemma 4.2) for $K$-invariant functions on $G/\Gamma$
then gives  us with $n>{s\over 2}$ 
the bound 
$$ \| \phi_\pi\|_{B_{|\pi|}, \infty} \leq  C ( 1+| \pi|)^ d \|\phi_{ \pi}\|_{2, n} \leq C 
\| \phi_\pi\|_{2, n+ 2d}\,,$$
again by use of (\ref{Sobinf}). 
Thus we obtain from (\ref{b1}) that 
\begin{eqnarray}  \notag\| \phi_d^H\|_p
&\leq& C \Big( \sum_\pi   \| \phi_\pi\|_{2, m+n +l+ 2d}^2  \Big)^{1\over 2}\leq C \| \phi_d\|_{2, m+n+l+2d} \\
 \label{b2}&\leq& C \| \phi\|_{2, m+n+l+2d} \leq 
C \| \phi\|_{\infty, m+n+l +2d}\, .
\end{eqnarray}

With  $k=m+n$ this shows the asserted bound for the discrete  spectrum. 
\par We move on to the continuous part $\phi_c$ of $\phi$.  
We recall the Plancherel formula for $\Hc_P$ from (\ref{L-PT})  and note that $\phi_c=\sum_{P\in{\mathcal P}\atop P \neq G}  \phi_P$.  Accordingly we have $\phi_P=\sum_{\sigma \in \hat M_{P,d}} \phi_{P,\sigma}$ where $\phi_{P,\sigma}$ refers to the $\sigma$-isotypical part. 
Fix now $\sigma\in \hat M_{P,d}$ and unitary model $V_\sigma$ of it. 
We choose  $\Hc_{\sigma,\lambda}:=L^2(K\times_{M_P\cap K}  V_\sigma)$ as unitary model for the principal series $\pi_{\sigma, \lambda}$. Note that this Hilbert model is uniform in $\lambda\in i\af_P^*$. 

Let $r_P:=\dim \af_P$.  Note that 
$|\pi_{\sigma, \lambda}|\asymp |\sigma|+ |\lambda|^2$. Hence 
for $k>{r_P\over 2}$ there exists a constant $C>0$ such that 
\begin{equation}\label{mu_c}
\int_{i(\af_P^*)^+} (1+|\pi_{\sigma, \lambda}|)^{-k} \ d\lambda
\leq \frac C{(1 +|\sigma|)^{k -\frac{r_P}2}} \, .
\end{equation}

Let 
\begin{equation} \label{phi min} \phi_{P,\sigma}=\int_{i(\af_P^*)^+} \phi_{\sigma,\lambda} \ d\lambda\end{equation}
be the spectral decomposition of $\phi_{P,\sigma}$.  Here $\phi_{\sigma, \lambda}$ are smooth $K$-invariant functions on $Y$ obtained from $G$-equivariant embeddings 
$$  \Hc_{\sigma, \lambda} \to C^\infty(Y), \ \ v_{\sigma, \lambda}\mapsto 
\phi_{\sigma,\lambda}$$
which satisfy 
$$ \|\phi_{P,\sigma}\|^2 = \int_{i (\af_P^*)^+}  \|v_{\sigma,\lambda}\|^2 \ d\lambda\, .$$

\par We recall from Theorem 3.2 of  \cite{B} specialized to $Y=G/\Gamma$  (see \cite{B}, Ex. 4.3.4)
that there exists a constant $C>0$ such that 
\begin{equation} \label{IB1}  \sup_{y\in Y} |\phi_{\sigma,\lambda}(y)  \vb_Y(y)^{1\over 2} ( 1 + |\log \vb_Y(y)|)^{-r-1}|\leq C \| v_{\sigma, \lambda}\|_{n}\qquad (\lambda\in i(\af_P^*)^+)\end{equation}
where $\|v_{\sigma, \lambda}\|_n$ refers to the $n$-th Sobolev norm of $\Hc_{\sigma, \lambda}$ with $n>{s\over 2}$ as before.
Here we are allowed to use the improved Sobolev shift (compared to $n>{\dim G\over 2}$ used in \cite{B}) 
as $\phi_{\sigma, \lambda}$ is  $K$-fixed.  For a short proof of (\ref{IB1}) in the very similar 
context with $Y$ replaced by $Z=G/H$ we refer 
to Lemma 9.4 in \cite{KS2}. 

\par It follows from (\ref{IB1}) that the integral (\ref{phi min}) is uniformly convergent on compact subsets of $Y$. Hence  we obtain from the compactness of 
$Y_H$ and Fubini's theorem
that
$$\phi_{P,\sigma}^H =\int_{i(\af_P^*)^+} \phi_{\sigma,\lambda}^H \ d\lambda\, .$$
From that we deduce that 
\begin{equation} \label{min1} \|\phi_{P,\sigma}^H\|_p\leq \int_{i(\af_P^*)^+} \|\phi_{\sigma,\lambda}^H\|_p \ d\lambda \, .\end{equation} 
Using (\ref{hypoB1}) and (\ref{hypoB2}) we obtain as in the discrete that

\begin{equation}\label{phi-P1} \|\phi_{\sigma,\lambda}^H\|_p\leq C \|\phi_{\sigma, \lambda}^H\|_{\infty, D_{d \log (1 +|\pi_{\sigma,\lambda}|)}, l}\leq C \|\phi_{\sigma, \lambda}\|_{\infty, B_{|\pi_{\sigma, \lambda}|}, l}\, .\end{equation} 
\par Now (\ref{phi-P1})  combined with (\ref{IB1})  and (\ref{S-bound})  yields that 

$$ \|\phi_{\sigma, \lambda}^H\|_p \leq C \|v_{\sigma, \lambda}\|_{s +l}  ( 1+|\pi_{\sigma,\lambda}|)^d \leq C \|v_{\sigma, \lambda}\|_{n+l+2d}\, .$$ 
Inserting this bound into (\ref{min1}) we arrive at 

\begin{eqnarray*}  \|\phi_{P,\sigma}^H\|_p &\leq&C  \int_{i(\af_P^*)^+} \|v_{\sigma, \lambda}\|_{n+l+2d}  \ d\lambda\\ 
&=&C \int_{i(\af_P^*)^+} \|v_{\sigma,\lambda}\|_{n+l+2d} (1 +|\pi_{\sigma, \lambda}|)^{k/2} ( 1+ 
|\pi_{\sigma, \lambda}|)^{-k/2}  \ d\lambda\\
&\leq&  {C\over (1 +|\sigma|)^{k\over 2}} \Big(\int_{i(\af_P^*)^+} \|v_{\sigma,\lambda}\|_{n+l+2d+{k\over 2}}^2  \ d\lambda\Big)^{1\over2}=
{C\over (1 +|\sigma|)^{{k\over 2} - \frac{r_P}2}} \|\phi_{P,\sigma}\|_{2, n+l +2d +{k\over 2}}\\ 
&\leq&  {C\over (1 +|\sigma|)^{ {k\over 2} - \frac{r_P}2}} \|\phi\|_{2, n+l +2d +{k\over 2}} 
\leq {C\over (1 +|\sigma|)^{{k\over 2}-\frac{r_P}2}}  \|\phi\|_{\infty, n+l +2d +{k\over 2}} \, .\end{eqnarray*}

From this inequality we obtain that

$$ \|\phi_P^H\|_p\leq \sum_{\sigma\in \hat M_{P,d}} \|\phi_{P,\sigma}^H\|_p
\leq C    \|\phi\|_{\infty, n+l +2d +{k\over 2}} \sum_{\sigma\in \hat M_{P,d}} {1 \over (1 +|\sigma|)^{{k\over 2}- \frac{r_P}2}} \, .$$ 
Finally for ${k\over 2}=m$ as above for the discrete spectrum, the sum 
$\sum_{\sigma\in \hat M_{P,d}} {1 \over (1 +|\sigma|)^{{k\over 2}- \frac{r_P}2}}$ is finite.
This completes the proof. 
\end{proof} 

Let $y_0=\Gamma\in Y$ be the standard base point and $\vb_Y$ be an associated volume weight. 

\begin{lemma} \label{Siegel lemma} Let $\Omega\subset G$ be a compact subset. Then there exists a constant $C>0$ such that 
$$\inf_{g\in\Omega} \vb_Y (ag\cdot y_0) \geq C a^{ 2\rho_\nf}  \qquad (a\in A^-)\, .$$ 

\end{lemma}

\begin{proof}  Let $B_N$ be a convex compact neighborhood of $\1$ in $N$, 
and likewise $B_A\subset A$.   Let $B_{\oline N} =\theta (B_N)\subset \oline N$.  Then $B:= B_{\oline N} 
B_A M B_N$ is a compact neighborhood of $\1$ in $G$, 
so that we may assume the volume weight is given by
$\vb_Y(y)= \vol_Y(By)$ for $y\in Y.$

Note that 
$$\Ad(a)^{-1}B_N\supset B_N\quad\text{and}\quad
\Ad(a)^{-1}B_{\oline N} \subset B_{\oline N}$$
for 
$a\in A^-$. Set $B_{\oline N}^a:= \Ad(a)^{-1}B_{\oline N}$. 
 Now for $y=ag\cdot y_0$ with $a\in A^-$ and $g\in \Omega$ we note that 
 $$\vol_Y(Bag\cdot y_0)=\vol_Y(a^{-1}Ba g \cdot y_0)\geq \vol_Y(B_{\oline N}^aMB_A B_Ng\cdot y_0)\, .$$
 Set $B^a:=B_{\oline N}^aMB_A B_N$ and observe $B^a\subset B$.  As $B\Omega$ is compact
 we thus obtain for all $a\in A^-$ and $g\in \Omega$
 $$\# (\Gamma\cap B^ag) \leq \# (\Gamma\cap B\Omega)=:f<\infty\, .$$
 
 It follows that 
 $$\vb_Y(ag\cdot y_0)\geq \vol_Y(B^ag \cdot y_0) \geq {1\over f}  \vol_G(B^ag)={1\over f} \vol_G(B^a)\, .$$
Finally,  $\vol_G(B^a)\geq C a^{2\rho_{\nf}}$ by the integral formula for the Bruhat decomposition. 
\end{proof}

\section{Proof of hypothesis B in case $\rank_\R Z=1$} \label{Sect6}

In this section we assume that $Z$ is wavefront with $\rank_\R Z =1$, i.e. $\dim \af_Z=1$
(see Remark \ref{RR1} for examples, and note in particular the cases (20) - (21), which are far from
being symmetric). 

\subsection{Basic geometry of rank one wavefront spaces}
We begin by showing that $\af_Z^-$ is a half-line. 

\begin{lemma} Let $Z$ be a wavefront real spherical space of real rank one. Then $\af_Z^-\neq \af_Z$. 
\end{lemma}

\begin{proof}  We argue by contradiction and assume that $\af_Z^-=\af_Z$. This implies $\af_Z$ equals
the edge of the compression cone and thus normalizes $H$.  The overgroup $\hat H := HA_Z$ is 
real spherical and the associated real spherical space $\hat Z=G/\hat H$ has real rank zero. 
But then \cite{KK}, Cor.~8.5,  implies that $\hat H$ contains a conjugate of $AN$ and this excludes
$H$ from being wavefront. \end{proof} 

Set $\af_Z^+:= -\af_Z^-$ and note that $\af_Z=\af_Z^- \cup \af_Z^+$ with $\af_Z^-\cap \af_Z^+=\{0\}$.

\par   The first important  feature 
of rank one spaces  is a simplified polar decomposition, namely:  there exists a compact subset 
$\Omega_Z\subset Z$ such that 

\begin{equation} \label{polar r1} Z = K A_Z^- \W \cdot z_0 \cup \Omega_Z\, .\end{equation}
The reason for that is that the  standard  compactification of $Z$ has only one $G$-orbit in the boundary
see \cite{KK}, Th. 13.7.

In particular we obtain for any non-negative measurable function $f$ on $Z$ that 

$$ \int_Z  f(z) \ dz   \leq  \int_{\Omega_Z} f(z)  \  dz + \int_{KA_Z^-\W\cdot z_0}  f (z) \ dz \, .$$  
We obtain from Lemma \ref{lemma int-ineq} (\ref{intineq2}) that 

$$   \int_{KA_Z^-\W\cdot z_0}  f (z) \ dz \asymp \sum_{\sw\in \W} \int_K \int_{A_Z^-} f(k a\sw\cdot z_0)  a^{-2\rho}    \ dk\  da $$
where $\asymp$ signifies that the ratio of the left hand side by the right hand side 
is bounded from above and below by positive constants independent of $f$.

In particular, if in addition $f$ is $K$-invariant then we obtain that 

\begin{equation} \label{Int3} \| f\|_1  \asymp \|f\|_{1, \Omega_Z}  +  \sum_{\sw\in W}  \int_{A_Z^-}  f (a\sw \cdot z_0)a^{-2\rho}  \  da\end{equation} 
As for proving Hypothesis B it is then not serious to assume that $Z=KA_Z^- \W \cdot z_0$. 
\par 
We now come to the main technical tool for the verification of Hypothesis B. 

\subsection{The constant term approximation}

Our concern then is with matrix coefficients $f=m_{v,\eta}$ for a unitarizable irreducible 
Harish-Chandra module $V$ and $v \in V^K$ a $K$-fixed vector.  

\par In general, if $V$ is Harish-Chandra module with smooth moderate growth completion $V^\infty$, 
then we refer to $(V,\eta)$ as a {\it spherical pair} provided that $\eta: V^\infty \to \C$ is a continuous 
$H$-invariant functional.

\begin{theorem}[Constant term approximation]\label{thm const}  Let $Z=G/H$ be a 
 wave-front real spherical space of real rank one. There exist constants $C, c_0>0$ with the following properties.
 
 Let $(V,\eta)$ be a  spherical pair with $V$ irreducible,
 which satisfies an priori-bound 
  \begin{equation} \label{a priori}  |m_{v,\eta}  (\omega a \sw ) | \leq  a^{r \rho}  p(v)    
 \qquad   (v \in V^\infty; \omega\in \Omega, a\in A_Z^-, \sw\in \W) \end{equation} 
for some $r>0$ and a $G$-continuous norm $p$.

Then there exists a number $\mu \in \C$ such that the following holds.
Let  $$\{ \lambda_1, \lambda_2\}= \{ \rho\pm\mu\rho\}\subset \af_{Z,\C}^*.$$
Then for all $v\in (V^\infty)^ M$ and $\sw\in\W$  
there exist $c_1(v,\sw),c_2(v,\sw) \in \C$ such that 
\begin{equation} \label{c-approx} | m_{v,\eta} ( a \sw \cdot z_0) -  {\bf const}_\sw(v) (a) |\leq 
C  a^{(r+ c_0)\rho} p_8(v)  \qquad  (a \in A_Z^-, \sw\in \W)\,  \end{equation} 
where
 \begin{equation} \label{def const} {\bf  const}_\sw(v)(a):=  
  \begin{cases} 
c_1 (v,\sw) a^{\lambda_1} + c_2(v,\sw)  a^{\lambda_2}  
&  \lambda_1, \lambda_2 \neq \rho\\ 
a^{\rho}(   c_1(v,\sw)  +  c_2(v,\sw)  \rho(\log a))  &  \lambda_1=\lambda_2=\rho
  \end{cases} 
      \qquad  (a \in A_Z^-).
 \end{equation}
Moreover, let $I_0:=\{i=1,2\mid c_i(v,\sw)\neq 0 \text{ for some }v,\sw\}$. Then
\begin{equation} \label{lambda ineq with r}
\re \lambda_i(X) \geq r \rho(X),\quad (X\in {\af_Z^+}, i\in I_0),
\end{equation}
and   
\begin{equation} \label{c-approx improved}  |{\bf const}_\sw(v)(a)|  \leq  C a^{r'\rho}  p_8(v)    
\qquad   (v \in V^M;  a\in A_Z^-, \sw\in \W) \, \end{equation}
for every $r\leq r'\leq r+c_0$
such that $\re \lambda_i(X) \geq r' \rho(X)$ for all $X\in \af_Z^+$, $i\in I_0$.

Finally, if $V$ is unitarizable, then $\mu \in \R \cup i\R$. 
\end{theorem}

\begin{proof}  In what follows, the elements of $\W$ can be dealt with on an equal 
footing and for our simplified discussion here we shall assume $\W=\{\1\}$. 
Furthermore, since all  ideas are contained in the example of $\gf=\sl(2,\R)$
and $\hf$ one-dimensional, we restrict ourselves here to that case. However, in Remark \ref{general rank one}
below we indicate the necessary changes for the general rank one case.

Let 
$$ X = \begin{pmatrix} 1 & 0 \\ 0 & -1\end{pmatrix},  \quad E = \begin{pmatrix} 0 & 1 \\ 0 & 0 \end{pmatrix},  \quad 
F = \begin{pmatrix} 0 & 0 \\ 1 & 0\end{pmatrix}\, .$$
We assume that $\af_Z=\R X$ and $\pf=\Span\{X,E\}$. Then $X \in \af_Z^+$.
Let 
\begin{equation}\label{a1} \phi (t) = m_{v, \eta}(\exp(-tX))\end{equation} 
for $t \in \R$.  

\par The Casimir element of $\gf$ is given by 
$$ \Cc = {1\over 2} X^2 + EF + FE = {1\over 2} X^2 + X + 2 FE\, . $$
In general for $Y\in \gf$ and $\psi \in C^\infty(Z)$ we set 
$(L_Y \psi) ( z) ={d\over ds}\Big|_{s=0}  \psi (\exp(-sY)z)$ 
for $ z\in Z$.

\par Let now $c\in \C$ be such that $2\Cc$ acts as $c \operatorname{id}_V$ on $V$.   
Set 

\begin{equation} \label{a2} \Phi(t)=\begin{pmatrix} \phi (t)\\ \phi '(t) \end{pmatrix} \, . \end{equation}
Now note with $f=m_{v,\eta}$ that 
$$ \phi''(t) = (L_{X^2} f) (\exp(-tX)) = (L_{2 \Cc -2X - 4FE} f )(\exp(-tX)) = c \phi (t) - 2 \phi'(t)  + r(t)$$
with 
$$r(t) = - 4(L_{FE} f) (\exp(-tX))\,  .$$
For any smooth function $\psi $ on $Z$, viewed as a
 right $H$-invariant function on $G$   we have for $a\in A_Z^-$ that 
\begin{equation} \label{a3}(L_F \psi)  (a) = - a^\alpha  (R_F \psi) (a)\end{equation}
with  $\alpha$ the positive root and $(R_F \psi)(a) = {d\over dt}_{t=0} \psi (a \exp(tF)) $. 
Let now $\hf = \R Y $  and note that $\gf=\hf +\pf$ implies that $Y$ can be normalized such that 
$F = Y +  d_1 X + d_2 E$ with $d_1, d_2\in\R$.    Thus we get 
\begin{equation} \label{a4} (L_F \psi) (a)  = d_1 a^{\alpha}  (L _{X} \psi) (a) + d_2 a^{2\alpha} (L_{E}\psi)(a)\, .\end{equation}
We let now $\psi=L_Ef$
and thus   obtain from    (\ref{a priori})  that 
\begin{equation}  \label{remainder r(t)}  |r(t)|\leq  C e^{- t (r+c_0)}  p_2(v) ,\qquad t\ge0,\end{equation}
for a universal constant $C\geq 0$ and $c_0>0$ (specifically for $\gf=\sl(2,\R)$ we have 
$c_0 = 2$  if $d_1\neq 0$ and otherwise $c_0=4$,
but in general it can be different).
 
Set $$R(t):= \begin{pmatrix} 0 \\ r(t)\end{pmatrix}.$$ 
Then our discussion shows that $\Phi$ satisfies the first order differential equation 

\begin{equation} \label{ODE}  \Phi' (t) =  \sA \Phi(t)  + R(t)\, , \end{equation} 
where 
$$ \sA=\begin{pmatrix}   0 & 1 \\ c & -2\end{pmatrix}$$

The characteristic polynomial of $\sA$ is given by $\lambda^2 + 2\lambda -c $ and thus has solutions
\begin{equation} \label{spectrum} \lambda_{\pm}=  -1 \pm \mu \end{equation}  
where 
\begin{equation} \label{mu value} \mu =\sqrt{ 1+c}\, .\end{equation} 
The general solution formula for (\ref{ODE}) then is 

$$ \Phi(t) =  e^{t \sA} \Phi(0)+  e^{t\sA} \int_0^t e^{-s\sA} R(s) \ ds\, .$$  
  
Note that $\sA$ is semi-simple if and only if $\mu\neq0$, i.e. $c\neq  -1$.  
For $\lambda\in \Spec \sA
=\{ \lambda_-, \lambda_+\}$ we write $E(\lambda)\subset \C^2$ for the corresponding  generalized eigenspace.   Let 
$0 <\delta<c_0 +r$ and write 
$$\sP: \C^2\to \bigoplus_{\lambda\in \Spec \sA\atop 
-\re \lambda \leq r + c_0-  \delta}  E(\lambda)$$
for the projection along the complementary  generalized eigenspaces.   
Next we note that 

\begin{eqnarray*}  \int_0^t  e^{-s\sA}  R(s) \ ds & =&   \int_0^t  e^{-s\sA}   \sP  R(s) \ ds+  \int_0^t  e^{-s\sA} (\1 - \sP)  R(s) \ ds\\ 
& =&   \int_0^\infty  e^{-s\sA}   \sP  R(s) \ ds  -   \int_t ^\infty  e^{-s\sA}   \sP  R(s) \ ds
+ \int_0^t  e^{-sA} (\1 - \sP)  R(s) \ ds \end{eqnarray*}  
 in which convergence is ensured by the definition of $\sP$, as $\delta>0$.

Set

\begin{eqnarray*}
I_1(t) & = & e^{t \sA} \int_t^\infty  e^{-s \sA}   \sP R(s) \ ds = 
\int_0^\infty  e^{ -s \sA} \sP R (s+t) \ ds \\ 
I_2(t) & = & e^{t \sA} \int_0^t  e^{-s \sA}  (\1 -  \sP) R(s) \ ds  = \int_0^t e^{ s \sA} ( \1 - \sP) R(t-s) \ ds 
\end{eqnarray*} 
This gives us 

\begin{equation}\label{I1I2}  \Phi(t) = e^{t \sA} u  - I_1(t) + I_2(t)\end{equation} 
for $u\in \C^2$  
given by 
\begin{equation} \label{u-eq} u = \Phi(0) + \int_0^\infty   e^{-s\sA} \sP R(s) \ ds\, .\end{equation} 
Next we estimate $I_1$ and $I_2$ for $t\geq 0$.  
For that we first recall the Gelfand-Shilov estimate for $\| e^{t \sA}\|$ for a complex $N\times N$-matrix 
with $\sigma:= \max \{ \re \lambda\mid \lambda \in \Spec \sA\}$:

$$ \| e^{t\sA}\| \leq e^{\sigma t} ( 1 + t \| \sA\| + \ldots +  { t^{N-1} \over (N-1)!}  \| \sA\|^{N-1})\qquad (t\geq 0)\, .$$
and in particular for $N=2$: 
$$ \| e^{t\sA}\| \leq e^{\sigma t} ( 1 + t \| \sA\|)\qquad(t\geq 0)\, . $$

Then,

\begin{eqnarray*}  \| I_1(t)\| & \leq &  \| \sP \|   \int_0^\infty e^{ s( r + c_0 -\delta)}   (1+ s \|\sA\|) \| R(s+t)\|  \ ds \\ 
& \leq &  C \| \sP \|    p_2(v)  \int_0^\infty e^{ s( r + c_0 -\delta)}  ( 1 + s \| A\| )    e^{ - (r + c_0) (s+ t)}   \ ds \\ 
& \leq &  {C \over \delta^2}  \|\sP\| ( 1 + \| \sA\|)    p_2(v) e^{ - (r + c_0) t} \\
& \leq &  {C \over \delta^2} \|\sP\|  p_4(v) e^{ - (r + c_0) t} \, . \end{eqnarray*} 
For the last line we used the fact, obtained from (\ref{Sobinf}),
that 
$$ (1 +\|\sA\|) p_2(v) \leq C (1+|c|)  p_2(v)\leq C p_4(v) \qquad (v\in V^\infty)\, .$$  
Further note that $C$ does not depend on $\delta$ as $\delta$ was bounded from above by $c_0 +r$.

 Similarly 
   \begin{eqnarray*}  \| I_2(t)\| & \leq &  \|\1- \sP \|   \int_0^t e^{ s ( -r - c_0 + \delta)}   (1+ s \|\sA\|) \| R(t-s)\|  \ ds \\ 
& \leq &  C \| \1 -\sP \|    p_2(v) e^{- (r + c_0)t}   \int_0^t  e^{ s\delta}  ( 1 + s \| A\| )      \ ds \\ 
& \leq &  {C \over \delta^2}  \|\sP\| ( 1 + \| \sA\|)    p_2(v) t e^{ - (r + c_0 -\delta) t}   \\
& \leq &  {C \over \delta^2}  \|\sP\|  p_4(v) (1+ t) e^{ - (r + c_0-\delta) t} \, . \end{eqnarray*} 

 With the Lemma \ref{matrix lemma} from below we can find a value
 $\frac {c_0}4\leq \delta \leq {c_0\over 2}$ 
 and a constant $C>0$ such that $\|\sP\| \leq C (1+\|\sA\|)^2$. This 
 yields then 
 \begin{equation} \label{I1I22} \| I_1 (t) \| + \| I_2(t)\| \leq     C e^{ - (r +{ c_0\over 2})  t}  p_8(v) 
 \, .\end{equation}
 
We now define ${\bf  const}_\sw(v)(\exp(-tX))$ to be the first coordinate of $e^{tA}u$.
By expanding $u=c_+ u_+  + c_- u_-$ into generalized eigenvectors of $\sA$ 
to eigenvalues $\lambda_{\pm}$ we obtain (\ref{def const}) for $\lambda_i=-\lambda_\pm\rho$.
Furthermore (\ref{c-approx}) follows from (\ref{I1I2}) and (\ref{I1I22}).

{}From (\ref{a priori}) and (\ref{c-approx}) it follows  that
\begin{equation*} |{\bf const}_\sw(v)(a)|  \leq  C a^{r\rho}  p_8(v)    
\qquad   (v \in V^M;  a\in A_Z^-, \sw\in \W) \, .\end{equation*}
Hence a coefficient $c_\pm$ vanishes provided that $\re \lambda_\pm > -r$, as stated in
(\ref{lambda ineq with r}).

{}From (\ref{u-eq}) we have $u=\Phi(0)+I_1(0)$ and hence $\|u\|\leq Cp_8(v)$.
If $\re \lambda_\pm \leq -r'$ this implies 
\begin{equation*}\label{ueq2} \|e^{t\sA}u\|\leq C e^{-tr'} p_8(v)\, \end{equation*} 
and hence  (\ref{c-approx improved}). 

Finally,  we remark that  if $V$ is unitary then the Casimir eigenvalue $c$ 
is real, and thus the final assertion is 
a consequence of (\ref{mu value}).
 \end{proof}

\begin{rmk}  \label{general rank one} The alert reader might ask where the assumption $v\in V^M$ entered the proof. For a 
general rank one space with parabolic $\qf=\lf +\uf$ we recall the shape of the Casimir operator $\Cc$.
For every root space $\gf^\alpha\subset\uf$ we let select a basis $(Y_\alpha^j)_{1\leq j \leq m_\alpha}$ 
with $m_\alpha= \dim \gf^\alpha$.  Then we choose $Y_{-\alpha}^j \in \gf^{-\alpha}$ such that 
$\kappa(Y_\alpha^j, Y_{-\alpha}^k)=\delta_{jk}$ with $\kappa$ the Cartan-Killing form. 
Set $\Cc_\uf:=\sum_{\alpha\in\Sigma(\af, \uf+\oline\uf)}  Y_{\alpha}^j Y_{-\alpha}^j$ and note that 
$$\Cc_\uf= \underbrace{2 \sum_{\alpha\in \Sigma(\gf,\uf)\atop
1\leq j\leq m_\alpha} Y_{-\alpha}^j Y_\alpha^j}_{=:\Cc_\uf^-}  + \underbrace{\sum_{\alpha\in \Sigma(\gf,\uf)\atop
1\leq j\leq m_\alpha} [Y_\alpha^j, Y_{-\alpha}^j] }_{=:X_{2\rho}\in\af}$$
We note that $2\rho= \kappa (X_{2\rho}, \cdot)$
as a functional on $\af$.  The fact that $Z$ is unimodular then implies that $X_{2\rho}\in \af_H^\perp$
by \cite{KKSS2}, Lemma 4.2.  Then 
$$\Cc=\Cc_{\lf}+\Cc_\uf$$ 
with $\Cc_{\lf}$ a multiple of the Casimir of $\lf$.  Now $\lf= \lf_{\rm n} \oplus  \mf_\lf \oplus (\zf(\lf)\cap \af)$
with $\mf_\lf<\mf$ and $\lf_{\rm n}<\lf\cap \hf$ by the local structure theorem.  In particular if $f$ is a left $M$-invariant function on $Z=G/H$ and $a\in A_Z$, then we have 

$$L_\Cc f  (a)=  L_{\Cc_{\zf(\lf)+\af} +\Cc_\uf}f (a) = L_{{1\over 2} (X_{2\rho})^2 + X_{2\rho} + \Cc_\uf^-} f(a) $$
and the analogy to the $\sl(2)$-case becomes apparent:  we define $\phi(t)$ as in (\ref{a1}) with $X=X_{2\rho}$
and obtain $\Phi$ as in (\ref{a2}).  The function $\Phi$ satisfies the ODE (\ref{ODE}) with $r(t)=-2 (L_{\Cc_\uf^-} 
f)(\exp(-tX_{2\rho})$. Further,  given the shape of $\Cc_\uf^-$ we can proceed as in (\ref{a3})-(\ref{a4}) to arrive at the 
remainder term estimate (\ref{remainder r(t)}). The rest of the proof translates verbatim.

\end{rmk}

\begin{lemma} \label{matrix lemma}Let $0<\nu\leq 1$, $N\in \N$ and $\sA\in \operatorname{Mat}_N(\C)$ with $\Spec (\sA)=\{ \lambda_1,\ldots, \lambda_r\}$
such that $\re \lambda_1\leq \ldots\leq \re \lambda_r$.  For every $1\leq j\leq r$ let $V_j\subset \C^n$ be the generalized 
eigenspace of $\sA$ associated to the eigenvalue $\lambda_j$.  For every $1\leq k\leq r$ we let $E_k=\bigoplus_{j=1}^k  V_j$
and $\sP_k : \C^N \to E_k$ be the projection along $\bigoplus_{j=k+1}^r V_j$.  Suppose for some 
$1\leq k\leq r-1$ that $\re \lambda_{k+1}- \re \lambda_k\geq \nu$.  
Then there exists a constant $C=C(\nu, N)>0$ such that 
$$\|\sP_k\|\leq C  (\|\sA\|+1)^N\, .$$
\end{lemma}

\begin{proof}  Let $R\subset \C$ be 
the positively oriented and axes-parallel rectangle which intersects the imaginary axis
in $\pm i(\|A\|+1)$, and the real axis in $-(\|A\|+1)$, respectively half way between
$\re \lambda_k$ and $\re \lambda_{k+1}$.

Observe that: 
\begin{itemize}
\item  $R$ surrounds $\{ \lambda_1, \ldots, \lambda_k\}$ but not $\{ \lambda_{k+1}, \ldots, \lambda_r\}$, 
\item $\operatorname{dist}(\Spec \sA, R)\geq \nu/2$, 
\item $|R| \leq 8\|\sA\|+8$. 
\end{itemize}

Next we recall that 
$$\sP_k = {1\over 2\pi i}\oint_R    (z-\sA)^{-1} \ dz \,.$$ 
and thus
$$\|\sP_k\| \leq {1\over 2\pi} |R|   \max_{z\in R} \| (\sA-z)^{-1}\|\,.$$
Cramer's rule gives 
$$ (\sA-z)^{-1}  = {1\over \det (\sA-z)}   ( (-1)^{i+j}\det( \sA -z)_{ij} )_{i,j}.$$
Using the observations above we get 
\begin{itemize}
\item $|\det (\sA-z)| \geq C_1$ for all $z\in R$ with $C_1=({\nu\over 2})^n$.
\item $|\det(\sA-z)_{ij}|\leq  C_2 (1+\|\sA\|)^{N-1}$ for all $z\in R$, $1\le i,j\le N$, and a  constant $C_2  = C_2(N)$
(apply Hadamard's inequality for the determinant). 
\end{itemize}
The assertion follows. 
\end{proof}

\subsection{Function spaces on $Z$}\label{subsection function spaces} 
The results in this subsection are generally valid for any unimodular real spherical space.

We first recall the two standard weight functions for $Z$, 
the volume weight $\vb$ (see Section~\ref{Sect1}) and the radial weight $\rb$   
(see \cite{KS2}, Sections 4 and 9),
with the bounds:
\begin{eqnarray} 
\label{v-bound} \vb(\omega a\sw\cdot z_0) &\asymp& a^{-2\rho}
\\
\label{r-bound} \rb(\omega a \sw\cdot z_0) &\asymp&   1 + \| \log a\|
\end{eqnarray} 
for all $\omega\in \Omega, a\in A_Z^-, \sw\in \W$.
The two weight functions are related to the invariant measure 
as follows (see \cite{KS2}, Lemma 9.4)
\begin{equation}\label{rb relation}   \int_Z  \vb(z)^{-1}  \rb(z)^{-s}  \ dz <\infty \qquad  (s> \rank_\R Z)\end{equation}
where $\rank_\R Z= \dim \af_Z$.

Given $1\leq p <\infty$, $m>0$ we consider the following norms on $C_c^\infty(Z)$:

$$\| f\|_{p, m}:=  \| f \rb^m\|_p= \Big(  \int_Z  |f(z)|^p  \rb(z)^m \ dz \Big)^{{1\over p}}$$
and 
$$ q_{p, m}(f)  =\sup_{z\in Z}    |f(z)|  \vb(z)^{1\over p}  \rb(z)^m\,.$$
If $k\in\N_0$  we denote by $\|\cdot\|_{p,m; k}$, resp.~$q_{p,m;k}$,
the $k$-the Sobolev norm of $\|\cdot\|_{p,m}$, resp.~$q_{p, m}$ (see (\ref{p2k})). 

 \begin{lemma} The two families of norms  $\|\cdot\|_{p,m; k}$ and $q_{p,m;k}$ are related as follows:
\begin{enumerate} 
\item \label{Eins}  For $pm' - m> \rank_\R Z$, there exists a constant $C=C(p, m, m',k)>0$  such that 
\begin{equation}\label{Sob1}  \| f\|_{p,m;k} \leq C  q_{p,m';k} (f)  \qquad (f \in C^\infty(Z)) \, .\end{equation} 
\item \label{Zwei} For $k' >  {\dim G\over p}$ there exists a constant $C=C(p,m, k, k')>0$ such that 
\begin{equation} \label{Sob2} q_{p, m;k}(f)  \leq  C \| f\|_{p, m;  k +k'}  \qquad   (f \in C^\infty(Z))\,.\end{equation}
\end{enumerate} 
\end{lemma} 

\begin{proof} (\ref{Eins}) For $k=0$ (no derivatives) this is an immediate consequence of  (\ref{rb relation}). The general case
is obtained by applying derivatives to $f$.  The inequality in (2) is established in the proof of \cite{KS2}, Lemma 9.5, for $p=2$.  The proof for general $p$ is entirely parallel. 
\end{proof}

In particular it follows from (\ref{v-bound}) and (\ref{Sob2}) for a matrix coefficient 
$f=m_{v,\eta}$ which is $L^p$-integrable that 

$$ \sup_{a\in A_Z^-\atop \sw\in\W} |f (a\sw\cdot z_0)|  a^{ -2\rho \over p} \leq C  \| f\|_{p; k} 
\qquad  ( k >{\dim G \over p})$$
with $\|\cdot\|_{p;k}:=\|\cdot\|_{p,0;k}$.  

Phrased differently all $L^p$-integrable
$f=m_{v,\eta}$ satisfy the a-priori bound 
\begin{equation} \label{a-priori2}   | f(a\sw\cdot z_0)| \leq   
C a^{r\rho}  \|f\|_{p;k}   \end{equation} 
for $r={2\over p}>0$ and the $G$-continuous norm  $\|\cdot\|_{p;k}$
as required in (\ref{a priori}).

\subsection{Verification of Hypothesis B}    In this subsection it is assumed that $Z$ is wavefront 
and $\rank_\R Z=1$. We also assume that $\W=\1$ for simplification of the exposition.
Furthermore, under the specific assumptions
in Theorem~\ref{thmC} we have for all non-trivial $\pi$ that 
$H_\eta/H$ is compact (see Remark \ref{compact quotient})
and thus we may as well assume that $H=H_\eta$. 

Let  $1\leq p'<p<\infty$.  
By Remark \ref{about p in Hypo B} 
it suffices to  establish the hypothesis for $p$ and $p'$ 
close to each other. We shall assume that 
$\frac 1{p'} -\frac 1p\leq \frac{c_0}2$  with $c_0$ 
the constant of Theorem \ref{thm const}.

Let $f=m_{v,\eta}$ be associated to a representation $\pi\in\Lambda_{p'}$. We assume that $f$ is $K$-fixed.
With the a priori bound (\ref{a-priori2}) for $r=\frac 2p$
we can apply Theorem \ref{thm const} and
let $\cb(f)(a)={\bf const}_v(a,\1)$ be a constant term for $f$ 
 (see (\ref{def const})). Then the 
approximation bound (\ref{c-approx})  reads

\begin{equation} \label{c-approx2}  |f(a) - \cb(f)(a)| \leq  C a^{(r +c_0)\rho}  \|f\|_{p; k+8} \qquad (a\in A_Z^-)
\, .\end{equation}

As $f$ is $K$-fixed we may assume with (\ref{Int3}) and the neglect of $\Omega_Z$  (on the fixed compact set 
$\Omega_Z$ the $L^p$-norm is dominated by the $L^\infty$-norm) that 

\begin{equation} \label{p-norm} \|f\|_p=\Big(\int_{A_Z^-} |f(a)|^p
a^{-2\rho} \ da \Big)^{1\over p}\, .\end{equation}
We now identify $A_Z^-$ with $[0,\infty[$ via an element $X\in \af_Z^-$ with $\rho(X)=-1$, i.e. 
$$[0,\infty[ \ni t \leftrightarrow a_t=\exp(tX)\in A_Z^-\, .$$
Then (\ref{p-norm}) translates into 

\begin{equation} \label{p-norm2} \|f\|_p=\Big(\int_0^\infty |f(a_t)|^p e^{2t} \  dt \Big)^{1\over p}\, .\end{equation}

We let $R>0$, to be specified later,  and begin with the estimate:

\begin{equation}  \label{e1} \|f\|_p = \| f\|_{p,[0,R]}  +  \|f\|_{p, [R,\infty[}  
\leq  e^{2R\over p} \| f\|_{\infty,[0,R]}+ 
\| f - \cb(f) \|_{p, [R,\infty[}  + \| \cb(f)\|_{p, [R,\infty[}\, .\end{equation}
By assumption $f$ is
$L^{p'}$-bounded (in particular satisfies the bound (\ref{a-priori2})
for $r={2\over p'}$), the leading coefficients $\lambda_i$ of $\cb(f)$ satisfy $\re \lambda_i(X)\leq -{2\over p'}$. 
Hence  we obtain from ${2\over p'} - {2\over p}\leq c_0$ and  
(\ref{c-approx improved}) that 
\begin{equation} \label{e2} \| \cb(f)\|_{p, [R,\infty[}  \leq  C \|f\|_{p; k+8}
\Big( \int_R^\infty e^{- 2{p\over p'}t}  e^{2t } dt\Big)^{1\over p} 
= C e^{-({2\over p'} - {2\over p})R} \|f\|_{p; k+8}\,.\end{equation}  
Note that $C$ does not depend on $\pi$.

Further, from  (\ref{c-approx2}) we obtain

\begin{equation} \label{e3}  \| f - \cb(f) \|_{p,[R,\infty[} \leq 
C\Big(\int_R^\infty e^{-(2 + pc_0)t} e^{2t}\,dt \Big)^{1\over p}   
\|f\|_{p;k+8}\leq C e^{-R c_0} \|f\|_{p;k+8}\, .\end{equation}
Set now  $\delta:={2\over p'} - {2\over p}$ and note that $\delta\leq c_0$.
Inserting the bounds from 
(\ref{e2}) and (\ref{e3}) into (\ref{e1}) we thus obtain that

$$ \|f\|_{p} \leq C (e^{ 2R\over p}  \|f\|_\infty+ e^{-R\delta} \|f\|_{p; k+8})$$
which, in view of (\ref{Sobinf}),  yields
\begin{equation}\label{e5}  \|f\|_{p} \leq C(e^{ 2R\over p}  \|f\|_\infty+  e^{-R\delta} (1+|\pi|)^{{k\over 2} + 4}  \|f\|_{p}) \,. \end{equation} 

So far, $R>0$ was arbitrary, but we now choose it such that

$$ C e^{- R\delta}  ( 1+|\pi|)^{{k\over 2}+4}\leq {1\over 2}$$
holds for the constant $C$ of (\ref{e5}), i.e. 

\begin{equation} \label{eR} R=R_\pi= {k+ 8\over 2\delta}  \log (1 +|\pi|) + \frac{\log(2C)}{\delta} \end{equation} 
will do. Then  we obtain from (\ref{e5}) 
$$\| f\|_{p} \leq 2Ce^{2R_\pi\over p} \|f\|_\infty,$$
and hence
$$
\| f\|_{p} 
\leq C ( 1+|\pi|)^{ {k+8\over \delta p}} \|f\|_\infty 
$$
with $C>0$ independent of $\pi$.
It follows  with $l:={k+8\over 2\delta p}$ 
\begin{equation} \label{e6} \|f\|_{p} \leq C \|f\|_{\infty; l}\,  \end{equation}
which is  Hypothesis B.

\section{$L^p$ versus $L^\infty$ in higher rank}\label{Sect7}
The purpose of this section is to  give a proof  of  (\ref{expect}) for general unimodular wavefront spaces. 
Observe that (\ref{expect}) is more general than Hypothesis B in two respects: 
\begin{itemize}  
\item The eigenfunction $f$ is not $K$-fixed.
\item The $G$-representation generated by $f$ is not necessarily unitarizable. 
\end{itemize} 

However,  our verification of Hypothesis B for the rank one cases does not really need
the assumption that the eigenfunction is $K$-fixed (the argument  given
allows to sum up $K$-types).   Further, the main ingredient is the 
constant term approximation, which does not require that the underlying Harish-Chandra module is 
unitarizable. 
 
\par Regarding the lattice counting problem we note that (\ref{expect})
also allows counting for balls which are not $K$-invariant.  

\par We now turn to the heart of the matter:

\subsection{The constant term approximation in higher rank}

Let $\tf\subset \mf$ be a maximal torus and set $\cf:=\af +\tf$. Then $\cf<\gf$ is a Cartan algebra and we let 
$\Wc_\cf$ be the Weyl group associated 
to the root system $\Sigma(\cf_\C, \gf_\C)$. Further we choose a positive system 
$\Sigma^+(\cf_\C, \gf_\C)$ of $\Sigma(\cf_\C, \gf_\C)$ such that $\Sigma^+(\cf_\C, \gf_\C)|_{\af} \bs \{0\} = \Sigma^+(\af, \gf)$. 

For a subset of spherical roots $I\subset S$ we set 
$$\af_I:=\{ X\in \af_Z\mid (\forall \alpha\in I) \   \alpha(X) =0\}\, ,$$
$$\af_I^{--}:=\{ X\in \af_I\mid (\forall\alpha \in S\bs I)\    \alpha(X) <0\}\subset \af_Z^-\, .$$

\par Finally for an irreducible Harish-Chandra module $V$ we let $\Lambda_V\in \cf_\C^*$ be the associated infinitesimal 
character. 
Then with the techniques provided in \cite{DKS}, but essentially the same proof, we obtain the following analogue 
of Theorem \ref{thm const}.

\begin{theorem}[Constant term approximation - higher rank]\label{thm const hr}   Let $Z=G/H$ be a 
wave-front real spherical space.
Let $I\subset S$ and $\Cc_I\subset \af_I^{--}$ be a compact subset.
Then there exist constants 
$C, c_0>0$, $k, d\in\N$  (depending only on $Z$ and $\Cc_I$) 
such that  the following assertion holds true: 

Let $(V,\eta)$ be a spherical pair with $V$ irreducible and 
which satisfies an a priori-bound 
 \begin{equation} \label{a priori hr}  
|m_{v,\eta}  (\omega a \sw\cdot z_0 ) | \leq  a^{r \rho}  p(v)    
 \qquad   (v \in V^\infty; \omega\in \Omega, a\in A_Z^-, \sw\in \W) 
 \end{equation} 
for some $r>0$ and a $G$-continuous norm $p$.    Then there exists
a subset $\E_V\subset \rho +\Wc_\cf\cdot \Lambda_V$ with  

\begin{equation}  (\forall \lambda\in \E_V, X\in \af_Z^-) \quad \re \lambda(X) \leq r \rho(X)  \end{equation} 
such that for all $v\in V^\infty$, $\sw\in\W$  and $\lambda\in \E_V$ there exist polynomials 
$c^I_{v,\lambda, \sw}\in \operatorname{Pol}(\af_I)$ of degree bounded by $d$ such that the sum

\begin{equation} \label{def const hr} {\bf  const}_\sw^I(v)(a):=  \sum_{\lambda\in\E_V} a^\lambda c^I_{v,\lambda, \sw}(\log a)\qquad (a\in A_Z)\end{equation}
satisfies for all $X\in\Cc_I$, $a\in A_Z^-$, $t\geq 0$, ${v\in V^\infty}$ and $\sw \in \W$ that

\begin{equation} \label{c-approx hr} | m_{v,\eta} ( a \exp(tX) \sw \cdot z_0) -  {\bf const}^I_\sw(v) (a\exp(tX) ) |\leq 
C  p_k(v)  a^{r\rho}    e^{  t (r+c_0) \rho(X)} \, .\end{equation} 
If in addition there exists an $r\leq r'\leq r+c_0$ such that the improved bound
\begin{equation}  (\forall \lambda\in \E_V, X\in \af_Z^-) \quad \re \lambda(X) \leq r' \rho(X)  \end{equation}
holds, then 
\begin{equation} |{\bf const}^I_\sw(v) (a) |\leq 
C  p_k(v)  a^{r'\rho}    \qquad(a\in A_Z^-, v\in V^\infty, w\in \W) \, .\end{equation}

\end{theorem}

\begin{rmk} \label{debate on ct}The constant term approximation in Theorem \ref{thm const hr} deviates from 
the one in \cite{DKS} in several aspects. The a priori-bound (\ref{a priori hr}) used in \cite{DKS}
 is 
\begin{equation} \label{a priori hr DKS}  |m_{v,\eta}  (\omega a \sw\cdot z_0 ) | \leq  a^\rho (1+\|\log a\|)^N  p(v)    
 \qquad   (v \in V^\infty; \omega\in \Omega, a\in A_Z^-, \sw\in \W) \end{equation} 
for some $G$-continuous norm $p$ and $N\in \N$. If we ignore logarithmic terms, then this is in essence the 
case $r=1$ above, and the presence of logarithmic terms does not pose additional technical difficulties. 
\par   The constant term in \cite{DKS} is unique and obtained from the one 
above by replacing $\E_V$ with $\E_V':=\{ \lambda\in \E_V\mid \re \lambda =0\}$, i.e. by discarding 
all exponents $\lambda$ for which $\re \lambda\neq 0$.  As such the constant term approximation in 
\cite{DKS} is only uniform if one fixes $\re \Lambda_V$.  The constant term approximation 
above however is uniform (i.e. holds for all $\Lambda_V$), but is non-unique 
and also allows exponents $\lambda$ with $\re \lambda\neq 0$.  
\end{rmk} 

We now provide a proof of Theorem \ref{thm const hr}, that is we will guide the reader through parts of  \cite{DKS} 
until the analogy to the rank one case becomes fully apparent.

\begin{proof}    We first recall the boundary degenerations $\hf_I$  of $\hf$ (see \cite{KKS}, Sect. 3):  Let $d:=\dim \hf$ and $X_I\in \af_I^{--}$.
Then 
$$\hf_I:= \lim_{t\to \infty}  e^{t \ad X} \hf$$
exists in the Grassmannian of $d$-dimensional subspaces of $\gf$ and is independent of the choice of $X_I\in \af^{--}$. 
The geometric meaning of $\hf_I$ is approximately as follows:   the homogeneous
space $Z=G/H$ admits a $G$-equivariant  compactification $\hat Z$ with $G$-boundary orbits  isomorphic to $\hat Z_I =  G/ \hat H_I$ where $\operatorname{Lie}(\hat H_I)= \hf_I + \af_I$.  The normal bundles to each $\hat Z_I$ are given by 
$Z_I=G/H_I$ where  $\operatorname{Lie}(H_I)= \hf_I$. 

Now choose $Y_I\in \af \cap \zf(\lf)$ such that $X_I = Y_I +\af_H$ and define a parabolic subalgebra

$$ \oline\pf_I=  \mf +\af +\bigoplus_{\alpha\in \Sigma\atop \alpha(Y)\geq 0}   \gf^\alpha\, .$$ 
We write $\oline\pf_I = \lf_I \ltimes  \oline \uf_I$  for the standard Levi-decomposition.   
A special feature of wavefront spaces then is (see \cite{KKS}, Prop. 5.5): 

$$ \oline \uf_I \subset \hf_I \subset \oline \pf_I\, .$$

\par Next we recall material from \cite{DKS}, Sect. 3.3.   Let $\Zc(\gf)$ and $\Zc(\lf_I)$ denote the centers of the universal 
enveloping algebras $\U(\gf)$ and $\U(\lf_I)$.   Denote by $\gamma_I:  \Zc(\gf) \to \Zc(\lf_I)$ the partial Harish-Chandra
homomorphism. We fix a finite-dimensional vector subspace $W\subset \Zc(\lf_I)$ such that 
$\Zc(\lf_I)\simeq \gamma_I(\Zc(\gf)) \otimes W$.  Recall that we assumed the Harish-Chandra module $V$ to be 
irreducible and we let $\chi_V: \Zc(\gf) \to \C$  the corresponding infinitesimal character.  Set ${\mathcal J}_{V,I}:= 
\gamma_I (\ker \chi_V)$ and record 
$$ \Zc(\lf_I)=  W  \oplus {\mathcal J}_{V,I} W\, .$$

\par  For a differentiable function $f$ on $Z$ and $X\in \gf$ we write $L_X f (z)={d\over dt}\big|_{t=0}  F(\exp(-tX)z)$
for the left derivative.    In case $f= m_{v,\eta}$ for a smooth vector $v\in V^\infty$ 
we record that $L_X (m_{v,\eta}) = m_{Xv, \eta}$. 
Now given $f=m_{v,\eta}$ as above we construct a $W^*$-valued  smooth function $F$ on $Z$ as follows: 

$$ F(z)(w):= L(w) f(z)  \qquad  (w\in W, z\in Z)\, .$$
Let $a\in A_Z^-$ and define a function  $F_{I,a}$ on $A_I$ by $F_{I,a}(a_I)=F(aa_I)$ for $a_I\in A_I$.

Then  there exists a Lie algebra homomorphism 

$$ \af_I \to  \operatorname{End} (W^*), \ \ X\mapsto \Gamma_V(X)$$ 
such that  $F_{I,a}$ satisfies the differential equation 

\begin{equation}\label{ODE hr} L_X  F_{I,a}  = \Gamma_V(X)  F_{I,a}   +   R_{I,a}   \qquad (X\in \af_I)\end{equation}
with 
$$R_{I,a}(a_I)(w)  =  \sum_{i=1}^n   L_{X_i} L_{v_i} f(aa_I)\qquad (a_I\in A_I)$$
where $(X_i)_{i=1}^n$ is a fixed basis of $\oline \uf_I$  and $v_i=v_i(w)\in \U(\gf)$ are elements which depend linearly 
on $w\in W$ and polynomially on $\chi_V$ (this is variant of formula  (3.23) in \cite{DKS}).  Moreover 
$\Gamma_V$ depends polynomially on $\chi_V$  and 
\begin{equation} \label{Spec Gamma}\Spec \Gamma_V(X) \subset -(\rho + \Wc_\cf \Lambda_V)(X)\, .
\end{equation} 

Recall the norm $p$ for which the matrix coefficients $m_{v,\eta}$ satisfy the a priori bound (\ref{a priori hr}).  Then we obtain 
from a slight variant of  \cite{DKS}, Lemma 3.5 (ii)  constants $C>0$ and $k, N\in \N$  which are all independent of $V$ such that

\begin{equation} \label{remainder R}\|R_{I,a}(\exp(X))\|\leq   C a^{r\rho} e^{ r\rho(X)   +\beta(X)}  (1 +|\chi_V|)^N  p_k(v) \qquad (X\in \af_I^-, a\in A_Z^-)\end{equation}
where $\beta(X):=\max_{\alpha\in \Sigma(\af, \uf_I)}   \alpha(X)<0$.

Also notice the general solution formula for (\ref{ODE hr})

$$F_{I,a}(\exp(tX))  =  e^{-t\Gamma_V(X)}  F(a)  -  e^{-t\Gamma_V(X)} \int_0^t e^{s\Gamma_V(X)} R_{I,a}(\exp(sX))
\ ds\qquad (X\in \af_I^-)\, .$$

\par Now we are ready to make the comparison to the proof in the rank one case.  The differential 
equation (\ref{ODE}) from the rank one case is replaced by (\ref{ODE hr}), 
and  (\ref{remainder R})
will replace the remainder estimate from (\ref{remainder r(t)}).  Further the equation (\ref{mu value})
for the exponents is replaced by (\ref{Spec Gamma}).  

\par As for the spectral gap we let 
$c_1$  be the largest possible value $c>0$  such that $c |\rho(X)| \leq |\beta(X)|$  for all $X\in\Cc_I$.  For $\delta>0$ we consider the spectral projection 
$$\sP:  W^*  \to \bigoplus_{\lambda \in \Spec \Gamma_V(X)\atop
-\re \lambda(X) \geq (r +c_1 -\delta)\rho(X)}  E(\lambda)$$
with $E(\lambda)$ denoting the generalized eigenspace of $\Gamma_V(X)$ on $W^*$. Because of 
(\ref{Spec Gamma}),  the pigeonhole principle combined with Lemma \ref{matrix lemma} yields a $\delta=\delta(V, \Cc_I)< 
{c_1\over 2|\Wc_\cf|}$ such that 
$$\|\sP\|\leq C (1 +|\Lambda_V|)^N\qquad (X\in\Cc_I) $$
for constants $C,N>0$ not depending on $V$ (recall that $\Gamma_V$ depends polynomially on $\Lambda_V$). This finally gives us 
$$\E_V:=\{  \lambda\in -(\rho +\Wc_\cf\Lambda_V)\mid (\forall X\in \Cc_I) \  (r+c_1-\delta) \rho(X) \leq  
-\re \lambda(X) \leq r \rho(X)\}\, .  $$

Having said all that we can now proceed 
as in the rank one case with integrals $I_1(t)$ and $I_2(t)$ now depending (uniformly) on $X\in \Cc_I$ and 
the additional variable $a\in A_Z^-$. 
This then leads to the asserted bound (\ref{c-approx hr}).  
\end{proof}

\subsection{Proof of (\ref{expect})}

We now prove (\ref{expect}) using the constant term approximation from Theorem \ref{thm const hr}.   The proof 
is parallel to the rank one case and we confine ourselves with an explanation of the  essential changes which are needed 
for the adaption of the proof. 

\par Let $\af_{Z,E}:=\af_S$ and note that $\af_{Z,E}$ 
normalizes $\hf$  (it is the Lie algebra of the "center"  of $Z$) and set $A_{Z,E}:=\exp(\af_{Z,E})$. 
Since the space of $H$-invariant continuous functionals of $V^\infty$ is finite dimensional \cite{KS},  we obtain 
an $A_{Z,E}$-finite right action on the eigenfunctions $f=m_{v,\eta}$ attached to $V$.  In particular there are no non-zero $L^p$-integrable eigenfunctions  if $\af_{Z,E}\neq \{0\}$.  We can thus request that $\af_{Z,E}=0$, and obtain 
in particular 
$$\af_Z= \af_I \oplus \af_{S\bs I} \qquad (I \subset S)\, .$$

We let $B$ be a compact neighborhood of $\1$ in $G$ such that
the conclusion of 
Lemma \ref{lemma int-ineq} (\ref{intineq1}) holds. In particular,

\begin{equation} \label{int-ineq} \int_Z  f(z)\ dz \leq  C \sum_{\sw\in \W} \int_B \int_{A_Z^-}  f(ga\sw\cdot z_0)  a^{-2\rho}  \ da \ dg \qquad 
( f\geq 0, \text{measurable})\end{equation}  
where $dz, dg, da$ are invariant measures on $Z, G, A_Z^-$. 

Next we need to decompose the compression cone $\af_Z^-$ suitably along its faces so that we can apply 
Theorem \ref{thm const hr}.   For $I\subset S$ and $X\in \af_Z$ we write $X=X_I +X^I$ with $X_I \in \af_I$ and $X^I \in 
\af_{S\bs I}$.  For $\delta>0$ and $\#S\bs I=1$ we define 
$$D_I:=\{ X\in \af_Z^- \mid  \|X\|\leq  (1 + \delta) \|X_I\|\} $$ 
and for $\#S\bs I>1$ by recursion 
$$D_I:=\{ X\in \af_Z^-\bs \bigcup_{S\supsetneq J\supsetneq I}D_J\mid  \|X\|\leq (1+\delta) \|X_I\|\}\, .$$ 
Then 
\begin{equation} \label{I-union} \af_Z^-=\bigcup_{I\subsetneq S} D_I\end{equation} 
and for every $I\subsetneq S$ there exists a compact set $\Cc_I \subset \af_I^{--}$ such that 
\begin{equation}\label{I-union2} D_I \subset \{X= X_I +X^I\in \af_Z^- \mid X_I \in \R_{\geq 0} \Cc_I, \|X\|\leq (1+\delta)\|X_I\|\}\, .\end{equation} 
Both assertions (\ref{I-union}) and (\ref{I-union2})  are geometrically obvious for $\dim \af_Z={1,2,3}$ and the general case is an elementary 
verification based on induction of the dimension.  For $X=X_I+X^I \in D_I$ and $a=\exp(X) \in A_Z^-$ we set 
$a_I=\exp(X_I)$ and $a^I=\exp(X^I)$ so that $a=a_I a^I$.

\par  Next we recall the notation on function spaces from Subsection \ref{subsection function spaces} and the norm-inequality in (\ref{Sob1}). 

Let $f=m_{v,\eta}$ be $L^p$-integrable with $v\in V^\infty$ a smooth vector.   For $g\in B, a\in A_Z^-, \sw\in \W$ now set 
$$\cb_\sw^I(f)(ga\sw\cdot z_0):={\bf const}_\sw^I(g\cdot  v) (a)\, .$$
Recall the compact set $\Cc_I\subset \af_I^{--}$ from (\ref{I-union2}). 
Then Theorem \ref{thm const hr} applied to $r={2\over p}$ implies a constant $c_0=c_0(V, \Cc_I)>0$ such that

\begin{equation} \label{c-approx2hr0}  |f(ga\sw\cdot z_0) - \cb_\sw^I(f)(ga\sw\cdot z_0)|\leq C  
a^{{2\over p}\rho}  a_I^{c_0 \rho} \|f\|_{p; k+k'}\qquad (g\in B, \log a\in D_I, \sw\in \W) \end{equation} 
where $k\geq 0$ is the Sobolev-constant from Theorem \ref{thm const hr}  and $k'>{\dim G\over p}$.   The constant $C$ is universal in the sense that 
it only depends on $p$ and $Z$. 
Now, (\ref{c-approx2hr0}) combined with (\ref{I-union2}) yields an $\e>0$,  only depending on the fixed choice $\delta>0$ 
and $c_0$,  such that 

\begin{equation} \label{c-approx2hr}  |f(ga\sw\cdot z_0) - \cb_\sw^I(f)(ga\sw\cdot z_0)|\leq C  a^{{(1+\e)2\over p}\rho} \|f\|_{p; k+k'}\qquad (g\in B, \log a\in D_I, \sw\in \W) \, .\end{equation}

\par For $R\geq 0$  we now set $D_{I, \geq R}:=\{ X\in D_I \mid - \rho(X) \geq R\}$ and set 
$$ \Omega_{\sw , \geq R}^I:=  B  \exp (D_{I,\geq R}) \sw\cdot z_0\, .$$
Further we set $\Omega_{\leq R}:= B A_R^- \W \cdot z_0$ and $\Omega_{\geq R}:= B (A^- \bs A_R^-) \W \cdot z_0$
and note that $\vol_Z (\Omega_{\leq R}) \leq C e^{2R}$ as a consequence of (\ref{int-ineq}).

\par The starting estimate (\ref{e1}) from the rank one case now becomes:

\begin{eqnarray}   \|f\|_p& =& \| f\|_{p, \Omega_{\leq R}}  +  \|f\|_{p, \Omega_{\geq R}} \notag \\ 
\label{E1}&\leq&  C e^{2R\over p} \| f\|_{\infty,\Omega_{\leq R}}+ 
\sum_{I\subset S\atop \sw\in \W} \big( \| f - \cb_\sw^I(f) \|_{p, \Omega_{\sw, \geq R}^I}  + \| \cb_\sw^I(f)\|_{p, \Omega_{\sw, \geq R}^I}\big) \, .\end{eqnarray}
The next two steps (\ref{e2}) and (\ref{e3}) from the rank one case are now completely parallel: Instead of (\ref{p-norm}) we use 

$$ \|f\|_p \leq  \Big(\sum_{\sw\in \W} \int_B \int_{A_Z^-}  |f(ga\sw\cdot z_0)|^p  a^{-2\rho} \ da\ dg \Big)^{1\over p}$$
which we obtain from (\ref{int-ineq}) and the 
a-priori estimate (\ref{c-approx2}) is replaced by the inequalities (\ref{c-approx2hr}).  We then arrive at the inequality 

\begin{equation} \label{E2} \|f\|_p \leq  C (e^{2R\over p}  \|f\|_\infty + e^{-R\e_0}\|f\|_{p; k +k'})\end{equation}
where $\e_0>0$ only depends on $\e>0$ and ${1\over p'} - {1\over p}>0$. 

So far we have not specified $R$ nor requested that $f$ is $K$-finite.   We now take the $K$-Fourier series of $f$: 

$$ f= \sum_{\tau\in \hat K} f_\tau$$
where 
\begin{equation} \label{f-tau} f_\tau(z):=  (\dim \tau) \int_K  \operatorname {tr} (\tau(k))  f( kz)   \ dk  \end{equation} 
Then we obviously have $\|f\|_p\leq \sum_{\tau\in \hat K} \|f_\tau\|_p$. The converse is true up to 
a Sobolev-shift $l=l(K)\in\N$, i.e. 

\begin{equation}\label{sum tau}   \sum_{\tau\in \hat K} \|f_\tau\|_p \leq C \|f\|_{p; l}\, .\end{equation} 
This is quite standard but we briefly sketch the proof for part of the audience which has not seen that before: 
Let $\chi_\tau\leq 0$ be the eigenvalue of $\Delta_\kf$ on the representation $\tau$.  Then, first it follows from  
(\ref{f-tau})  that  $\|f_\tau\|_p\leq (\dim \tau)^2 \|f\|_p$. Thus for $l\in \N$ we get 
$$ |\chi_\tau|^l \|f_\tau\|_p = \|\Delta_\kf^l  f_\tau\|\leq  (\dim \tau)^2\|f\|_{p; 2l}\, .$$
Finally the Weyl-dimension formula gives $\sum_{\tau\in \hat K}  {(\dim \tau)^2 \over |\chi_\tau|^l}<\infty$ provided 
$l=l(K)$ is large enough and (\ref{sum tau}) is established. 

\par Our reasoning implies in particular that we may assume in the sequel that $f=f_\tau$.  
Next we note that $\Delta f = (\chi_V + 2 \chi_\tau) f$ and therefore we get from (\ref{E2}) that 
\begin{equation} \label{E3} 
\|f\|_p \leq  C (e^{2R\over p}  \|f\|_\infty + e^{-R\e_0} ( 1 + |\chi_V + 2\chi_\tau|)^{k +k'\over 2}  \|f\|_{p})
\end{equation}
which is the analogue of (\ref{e5})  in the rank one case. From here the arguments run parallel to the rank one case.

\section{Appendix:  Mostow decomposition}

In this appendix we establish a variant of the Mostow-decomposition which we use in Section \ref{Sect2}.

\begin{theorem} \label{Mostow dec}{\rm (Mostow, \cite{MostowII})}Let $G$ be a real algebraic group and 
$H<G$ a real algebraic subgroup.  Then there exists 
a maximal compact subgroup $K<G$ such that $K\cap H $ is maximal compact in $H$ and a vector subspace
$V\subset \gf$ which is invariant under $N_K(H)$ such that the Mostow map 

\begin{equation}\Phi:  K \times_{K\cap H}  V \to G/H,  \ \ [g,X]\mapsto g \exp(X)H\end{equation} 
is a diffeomorphism. 
\end{theorem}

\begin{rmk}  The original statement of Mostow is slightly weaker and deviates from the one above in the following sense:  
Mostow shows that there is a subspace $V\subset \gf$ with $V=V_1 \times \ldots \times V_n$ 
a product of vector spaces 
and for $X=(X_1,\ldots, X_n)\in V$ he uses instead of $\exp(X)$ the expression 
$\Exp(X):=\exp(X_1)\cdot \ldots\cdot \exp(X_n)$. 
\end{rmk}

\begin{proof}  In the first step we assume that both $G$ and $H$ are reductive and let $\Theta$ be a Cartan-involution of $G$
which fixes $H$.  The Cartan involution gives us a maximal compact subgroup $K:=G^\Theta$ which has the property 
that $K\cap H$ is maximal compact in $H$.  Further let $\gf=\kf +\pf$ the decomposition into $\pm 1$ -eigenspaces 
of $\theta:=d\Theta(\1)$.  Then we obtain with $V:=\pf \cap \hf^\perp$ (where $\perp$ is taken w.r.t. to 
a fixed non-degenerate
invariant form on $\gf$)  that   $\Phi$ is a diffeomorphism (see \cite{Kob}, Lemma 2.7, for a short proof or 
\cite{HS}, Th. 9.3,  for a proof using gradient maps).  Note that $V$ is also 
$N_K(H)$-invariant.

\par In the second  step we consider the case where $G=L\ltimes U$ is a semi-direct product of a reductive group 
$L$ and a unipotent group $U$.  Further we assume that $H=L_H \ltimes U_H$ is a Levi-decomposition 
with $U_H<U$ and $L_H <L$.  We find a maximal compact subgroup $K<G$ which is contained in $L$ and such that 
$K\cap H$ is maximal compact subgroup of $H$.  From the first step applied to $L/L_H$ we find a 
$N_K(L_H)$-invariant subspace 
$V_1\subset \lf$ such that $\Phi_1:  K \times_{K\cap H} V_1 \to L/L_H$ is a diffeomorphism.   
By Lemma \ref{nil-Lemma} 
below we find a complementary subspace $V_2\subset \uf$ to $\uf_H$ which is $N_K(U_H)$-invariant 
such that $V_2\to U/U_H, \ X\mapsto \exp(X)U_H$ is a diffeomorphism. 
With $V:=V_1 \times V_2$ 
we finally claim that $\Phi$ is a diffeomorphism. 
To prepare the proof of the claim let us assume first that $\uf$ is abelian.
Then it follows
(cf.~\cite{Rossmann} exercise 7, p.~75, but note the sign error)
for $X_1\in\lf$ and $X_2\in\uf$  that
$$ \exp_G(X_1 +X_2)=(\exp_L(X_1), \exp_U({e^{\ad X_1} - \1\over \ad X_1} X_2))\in L \ltimes U\, .$$
If we assume $\uf_H \triangleleft \uf$ is a coabelian ideal, then 
we can apply the above to the semi-direct product $L\ltimes U/U_H$ and get for 
$X=(X_1, X_2) \in V=V_1 \times V_2$ 
$$ \exp_G(X_1+X_2)U_H =(\exp_L(X_1), \exp_U({e^{\ad X_1} - \1\over \ad X_1} X_2)U_H)\in L \ltimes U/U_H\, .$$  
Now $\ad X_1$ has real spectrum and therefore ${e^{\ad X_1} - \1\over \ad X_1} $ is 
invertible on the abelian Lie algebra $\uf/\uf_H$.  Hence from the fact that $\Phi_1$ is a diffeomorphism 
we obtain that $\Phi$ is a diffeomorphism as well, in case $\uf_H$ is a co-abelian ideal. 
 The general case is then obtained via induction on the filtration (see 
the proof of Lemma \ref{nil-Lemma}.)
Finally $V$ is $N_K(H)$-invariant since $N_K(H)=N_K(L_H)\cap N_K(U_H)$.

\par In the third step we consider the case where $G$ is reductive and $H$ is an arbitrary 
algebraic subgroup of $G$.
A well-known  theorem of Borel and Tits (see \cite{Hu}, Sect. 30.3, Cor. A) implies the existence of a
parabolic subgroup $P=L\ltimes U$ of $G$ with $P\supset H$ and such that $H=L_H \ltimes U_H$ is 
a Levi-decomposition of $H$ with $L_H=H \cap L$ and $U_H=U\cap H$.   By step 2 we obtain
a maximal compact subgroup $K_P\subset L$,
a $N_{K_P}(H)$-invariant vector subspace $V\subset \pf$, and
a diffeomorphism $\Phi_2:  K_P \times_{K_P\cap H}  V \to P/H$ of the desired form, 
where $K_P$ is an appropriate 
maximal compact subgroup of $P$. Let $K$ be a maximal compact subgroup 
of $G$ for which $K_P=K\cap P$.
In view of 
$$ G/H \simeq   G \times_P  (P/H)\simeq  K \times_{K_P} (P/H) 
\simeq K \times_{K_P} (K_P\times_{K_P \cap H} V)\simeq 
K \times_{K\cap H} V $$ 
we deduce the existence of $\Phi$ in this case as well. The Borel-Tits construction of $P\supset H$
(see \cite{Hu}, Sect. 30.3) entails that 
$P$ can be chosen such that it contains $N_G(H)$. It follows that $N_K(H)\subset K_P$,
and hence $V$ is $N_K(H)$-invariant.

\par Finally for most general case where $G=L \ltimes U$ is not necessarily reductive and $H<G$ real algebraic
we first consider the intermediate situation with $\tilde H = H U$ for which we obtain a 
$N_K(\tilde H)$-invariant subspace $\tilde V\subset \lf$ and a
diffeomorphism 
$\tilde \Phi:  K \times_{K_H} \tilde V \to G/\tilde H$ by step 3.   Further,  $\tilde H/H \simeq U/U_H$ for $U_H =U \cap H$. 
As in step 2 we choose a $K_H$-invariant complement $V_U$ to $\uf_H$ in $\uf$ and obtain a Mostow-map 
$\Phi$ with $V=\tilde V \times V_U$. 
\end{proof}

\begin{lemma} \label{nil-Lemma} Let $\uf$ be a nilpotent Lie algebra and $U$ the associated simply connected 
Lie group.  Let $\uf_H<\uf$ be a subalgebra and $U_H:=\exp(\uf_H)$.  Let $n\in\N$ be the nilpotency 
degree of $\uf$, i.e. $\uf^n=\{0\}$ and $n$ is minimal with respect to this property.   For $0\leq j\leq n-1$ we consider the subalgebras 
$\wf_j:=\uf_H + \uf^j$.  Then $\wf_{j+1}\triangleleft\wf_j$ is a co-abelian ideal. Further 
let $V_j\subset \wf_j$ be any complementary subspace to $\wf_{j+1}$ and set 
$V:=V_0 \oplus \ldots \oplus V_{n-1}$.  Then $\uf_H + V =\uf$ is a direct vector sum and the map 
$$ V \times U_H \to U, \ \ (X, h)\mapsto \exp(X)h$$ is 
a diffeomorphism.  If in addition $K_H< \Aut(\uf)\cap \Aut(\uf_H)$ is a reductive group of automorphisms,
then we can request each $V_j$ and hence $V$ to be $K_H$-invariant as well.
\end{lemma}

\begin{proof} (cf. Lemma 7.5 in \cite{KSS2}) Since all $\uf_H^{j+1}\triangleleft\uf_H^{j}$ are co-abelian, we obtain by induction on $n$ 
that the assignment 
$$(X,h)\mapsto \underbrace{\exp(X_0) \cdot\ldots \cdot \exp(X_{n-1})}_{:=\Exp(X)} h$$ 
is diffeomorphic. Here we expanded $X\in V=V_0 \oplus \ldots \oplus V_{n-1} $ according to the indicated 
factors. It remains to show that we can replace $\Exp(X)$ by $\exp(X)$.  In case $n=1$ this is clear
as $\exp(X)=\Exp(X)$. Moving by induction we obtain with $V':= V_1 \oplus \ldots \oplus V_{n-1}$ that the map 
$$ V' \times  U_H \to W_1, \ \ (X', h) \mapsto \exp(X')h$$
is diffeomorphic.  The induction step then follows from the observation that $\exp(X_0+ X') W_1=
\exp(X_0) W_1$ as $\exp:  \uf/ \wf_1 \to U/W_1$ is a homomorphism. 
 
\par For the  final assertion about $K_H$-invariance we note that $K_H$ preserves each $\wf_j$
and $\uf_H$. Hence $K_H$ preserves each $\wf_j$ and thus, by complete reducibility, we find 
for each $j$ a $K_H$-invariant complement $V_j$ to $\wf_j$ in $\wf_{j-1}$. 

\end{proof}

\end{document}